\documentclass[12pt, a4paper]{article}
\usepackage{amsmath}
\usepackage{amsmath,amssymb,amsthm,amsxtra} 
\usepackage{setspace}
\usepackage{amssymb}
\usepackage{dsfont}
\usepackage{enumerate}
\usepackage{float}
\usepackage{hyperref}
\usepackage{url}
\usepackage{cite}
\usepackage{hyperref}
\usepackage{graphicx}
\usepackage{epstopdf}
\usepackage{tikz}
\usepackage{geometry}
\newtheorem{theorem}{Theorem}[section]
\newtheorem{algo}{Algorithm}[section]

\newtheorem{lemma}[theorem]{Lemma}

\newtheorem{defn}[theorem]{Definition}
\newtheorem{definition}[theorem]{Definition}

\newtheorem{example}[theorem]{Example}

\newcommand{\ttt}{\texttt}
\newcommand{\RR}{\mathbb{R}}
\newcommand{\CC}{\mathbb{C}}
\newcommand{\ZZ}{\mathbb{Z}}
\newcommand{\QQ}{\mathbb{Q}}
\newcommand{\NN}{\mathbb{N}}

\newcommand{\PP}{\mathbb{P}}
\newcommand{\FF}{\mathbb{F}}
\newcommand{\HH}{\mathbb{H}}

\newcommand{\la}{\langle}
\newcommand{\ra}{\rangle}
\newcommand{\bsm}{\left ( \begin{smallmatrix} }
\newcommand{\esm}{\end{smallmatrix} \right )}
\newcommand{\bma}{\begin{pmatrix}}
\newcommand{\ema}{\end{pmatrix}}
\newcommand{\beq}{\begin{equation*}}
\newcommand{\eeq}{\end{equation*}}
\newcommand{\Ac}{\mathcal{A}}

\newcommand{\bc}{\mathcal{B}}
\newcommand{\Cc}{\mathcal{C}}
\newcommand{\cc}{\mathcal{C}}
\newcommand{\dc}{\mathcal{D}}
\newcommand{\ec}{\mathcal{E}}
\newcommand{\fc}{\mathcal{F}}
\newcommand{\gc}{\mathcal{G}}

\newcommand{\hc}{\mathcal{H}}
\newcommand{\ic}{\mathcal{I}}
\newcommand{\jc}{\mathcal{J}}

\newcommand{\lc}{\mathcal{L}}

\newcommand{\nc}{\mathcal{N}}

\newcommand{\pc}{\mathcal{P}}

\newcommand{\uc}{\mathcal{U}}

\newcommand{\Sc}{\mathcal{S}}

\newcommand{\opn}{\operatorname}

\newcommand{\op}{\oplus}

\newcommand{\bi}{\begin{itemize}}
\newcommand{\ei}{\end{itemize}}
\newcommand{\bpm}{\begin{pmatrix}}
\newcommand{\epm}{\end{pmatrix}}
\newcommand{\ol}{\opn{O}^+(L)}
\newcommand{\old}{\opn{O}^+(L')}
\newcommand{\otl}{\widetilde{\opn{O}}^+(L)}

\title{Orbits in lattices}
\author{Matthew Dawes}
\date{}
\begin{document}
\maketitle
\begin{abstract}
  Let $\opn{O}(L)$ be the orthogonal group of a lattice $L$.
  We exhibit algorithms for calculating Tits' buildings and orbits of vectors in $L$ for certain subgroups of $\opn{O}(L)$.
  We discuss how these algorithms can be applied
  to understand the configuration of boundary components in the Baily-Borel compactification of orthogonal modular varieties
  and to improve the performance of computer arithmetic of orthogonal modular forms.
\end{abstract}
\section{Introduction}
We exhibit algorithms for calculating Tits' buildings and orbits of vectors in a lattice $L$ for certain subgroups of the orthogonal group $\opn{O}(L)$.
Our motivations are twofold.
In one direction, we show how our results can be applied to understand the configuration of boundary components in the Baily-Borel compactification of orthogonal modular varieties.
Such questions are natural problems in algebraic geometry \cite{DawesFamily, abeliancompact, Scattone, Sterk}, but are typically addressed via ad-hoc methods.
We believe our results are the first broadly applicable results for orthogonal modular varieties.
In another direction, we show how our results can be applied to improve the performance of computer arithmetic of orthogonal modular forms, which can be useful in a variety of situations (e.g. calculating Fourier expansions of Borcherds products \cite{GKRBorcherdsproducts}).
While theoretical frameworks for fast multiplication have been discussed in the literature (e.g. \cite{RaumImplement}), practical implementations are only known for  certain special cases (e.g. for Hermitian modular modular forms using ideas based on Minkowski reduction \cite{DernThesis}) which do not generalise to the orthogonal group.

Our main results are Algorithms \ref{non_iso_def_alg}, \ref{non_iso_indef_alg0} and \ref{non_iso_indef_alg} for finding orbits of vectors in $L$ for certain subgroups of $\opn{O}(L)$ and Algorithms \ref{MaximalAlgo} and \ref{SuperGroupAlg} for calculating Tits' building of subgroups of $\opn{O}(L)$ when $L$ is of signature $(2,n)$.
We have produced software \cite{boundary} for calculating certain Tits' buildings using these algorithms: example calculations can be found in \S\ref{ExampleSection}.  
While our results are mostly intended for use with computer algebra packages such as \cite{GAP, SAGE, MAGMA}, we have interspersed a number of examples to show the algorithms can sometimes be used for manual calculation.
\subsection{Lattices}
Unless otherwise stated, a \emph{lattice} $L$ is an even, non-degenerate integral quadratic form on a free abelian group of finite rank \cite{Nikulin, SPLAG}.
We denote the bilinear form of $L$ by $(-,-)$ and, for a given $\ZZ$-basis $\{x_i \}$ of $L$, we let $G(L):=((x_i, x_j))_{i, j}$ denote the associated Gram matrix.
Examples of lattices include root lattices such as $A_n$ (we will assume all roots lattices are negative definite);
the rank 1 lattice $\la d \ra$ generated by a single element $x$ of squared length $x^2 := (x,x)=d$;
and the hyperbolic plane $U$, whose Gram matrix is given by
\begin{equation}\label{ugram}
G(U) = \bpm 0 & 1 \\ 1 & 0 \epm
\end{equation}
for a suitable basis (a \emph{canonical basis} for $U$).
We use $L_1 \op L_2$ to denote the orthogonal direct sum of lattices $L_1$ and $L_2$,
$mL$ to denote the orthogonal direct sum of $m$ copies of $L$
and $L(m)$ to denote the lattice obtained by multiplying the quadratic form of $L$ by $m$.
The \emph{signature} of $L$ is the pair $(t_+, t_-)$ consisting of the number of positive $t_+$ and negative $t_-$ squares in a diagonalisation of $G(L)$.
A lattice $L$ is  
\emph{(positive) definite} if $t_-=0$,
\emph{(negative) definite} if $t_+ = 0$
and \emph{indefinite} otherwise;   
$L$ is \emph{isotropic} if there exists $0 \neq x \in L$ such that $x^2=0$;  
a sublattice $S \subset L$ is \emph{totally isotropic} if the quadratic form of $L$ restricts to zero on $S$;
and a vector $0 \neq x \in L$ is \emph{isotropic} if $x^2 = 0$.
A sublattice $S \subset L$ is \emph{primitive} if $L/S$ is torsion-free and a vector $0 \neq x \in L$ is \emph{primitive} if  $\la x \ra \subset L$ is primitive.
The \emph{dual lattice $L^{\vee}$} of $L$ is the lattice defined on the group $\opn{Hom}(L, \ZZ) \subset L \otimes \QQ$ by the quadratic form of $L \otimes \QQ$.
For $x \in L$, we let $\hat{x}:={}^{\tau}xG(L)$. 
If $0 \neq x \in L$, we let  $\opn{div}(x)$ (the \emph{divisor} of $x$) denote the positive generator of the ideal $(x, L)$ and define $x^*:=x/\opn{div}(x) \in L^{\vee}$.
We will use $\opn{Iso}(L_1, L_2)$ to denote the set of isomorphisms between lattices $L_1$ and $L_2$ compatible with the quadratic forms of $L_1$ and $L_2$.
A lattice $L'$ is said to belong to the same \emph{genus} as $L$ if the quadratic forms $L' \otimes \ZZ_p \cong L \otimes \ZZ_p$ for all primes $p$ and $L' \otimes \RR \cong L \otimes \RR$.
If $\opn{Iso}(L, L')$ is non-trivial then $L$ and $L'$ are said to belong to the same \emph{class}.
We let $\opn{gen}(L)$ denote the set of classes in the same genus as $L$.
The (real) \emph{Witt index} of $L$ is defined as the maximal dimension of a totally isotropic subspace of $L \otimes \RR$ and, for prime $p$, the \emph{$p$-rank} $\opn{rank}_p(L)$ of $L$ is the maximal rank of all sublattices $S \subset L$ such that $\opn{det}(S)$ is coprime to $p$.
\subsection{Discriminant forms}
The \emph{discriminant group} of a lattice $L$ is defined as the abelian group $D(L):=L^{\vee}/L$ \cite{Nikulin}.
There is a natural $\QQ/2\ZZ$-valued finite quadratic form $q_L$ on $D(L)$ inherited from $L$ (the \emph{discriminant form} of $L$) \cite{Nikulin}.
It is known that the genus of $L$ is uniquely determined by the data $(t_+, t_-, q_L)$ \cite{Nikulin}.
For lattices $L_1$ and $L_2$, we let $\opn{Iso}(D(L_1), D(L_2))$ denote the set of group isomorphisms from $D(L_1)$ to $D(L_2)$, we define 
\beq
\opn{Iso}(q_{L_1}, q_{L_2}) := \{ g \in \opn{Iso}(D(L_1), D(L_2)) \mid 
q_{L_2} \circ g \equiv q_{L_1} \bmod{2 \ZZ} \}
\eeq
and use $\opn{O}(D(L))$ to denote $\opn{Iso}(q_L, q_L)$.
We use $[d_1, \ldots, d_{\opn{min}(m,n)}]_{m,n}$ to denote the $m \times n$ matrix whose $(i,i)$-th entry is $d_i$ and all other entries are zero.
We use $(y_1 \vert \ldots \vert y_m)$ to denote the matrix whose $i$-th column is given by $y_i$.
For $A \in M_{m,n}(\ZZ)$, there exist $P(A) \in M_{m,m}(\ZZ)$ and $Q(A) \in M_{n,n}(\ZZ)$ such that
\begin{equation}\label{discsmith}
P(A) A Q(A) = [d_1, \ldots, d_{\opn{min}(m,n)}]_{m,n}
\end{equation}
(the \emph{Smith normal form} of $A$) where, for some $r \in \NN$,  $d_i \neq 0$ for $i \leq r$, $d_i=0$ for $i>r$ and $d_i \vert d_{i+1}$ for $i=1, \ldots, r-1$ \cite{morris}.
For a fixed basis of $L$, the Gram matrix $G(L)$ defines an inclusion $G(L):L \hookrightarrow L^{\vee}$, allowing one to calculate the structure of $D(L)$ and a set of generators, using \eqref{discsmith}.
\subsection{The orthogonal group}
Let $\opn{O}(L)$ and $\opn{O}(L \otimes \FF)$ denote the orthogonal groups of $L$ and $L \otimes \FF$, where $\FF = \QQ$, $\RR$ or $\CC$.
Every $g \in \opn{O}(L \otimes \RR)$ can be written as a product of reflections  
\begin{equation}\label{refdec}
g = \sigma_{w_1} \ldots \sigma_{w_m}
\end{equation}
where $\sigma_w \in \opn{O}(L \otimes \RR)$ is defined by
\beq
\sigma_w: x \mapsto x - \frac{2(x,w)}{(w,w)}x \in \opn{O}(L \otimes \RR)
\eeq
for $w \in L \otimes \RR$ \cite{cassels}.
The \emph{spinor norm} $\opn{sn}_{\RR}(g)$ of $g$ as in \eqref{refdec} is defined by \cite{Kneser}
\begin{equation}\label{spinordef}
  \opn{sn}_{\RR}(g) =
  \left ( \frac{-(w_1, w_1)}{2} \right )
  \ldots
    \left ( \frac{-(w_m, w_m)}{2} \right )
\in \RR/(\RR^*)^2.
\end{equation}
We use $\opn{O}^+(L \otimes \RR)$ to denote the kernel of $\opn{sn}_{\RR}(g)$ in $\opn{O}(L \otimes \RR)$ and, for  $\Gamma \subset \opn{O}(L \otimes \RR)$, we let $\Gamma^+:=\Gamma \cap \opn{O}^+(L \otimes \RR)$.
There is a natural map 
\begin{equation}\label{OODL}
\opn{O}(L) \rightarrow \opn{O}(D(L)).
\end{equation}
We use $\overline{g}$ to denote the image of $g \in \opn{O}(L)$ under \eqref{OODL} and $\widetilde{\opn{O}}(L)$ to denote the kernel of \eqref{OODL} (the \emph{stable orthogonal group}). 
More generally, for $\Gamma \subset \opn{O}(L)$ and $\Ac \subset \opn{O}(D(L))$, we let
$\widetilde{\Gamma} := \Gamma \cap \widetilde{\opn{O}}(L)$
and
$\Gamma_{\Ac} := \{ g \in \Gamma \mid \overline{g} \in \Ac \}$.
\begin{lemma}[{\hspace{1sp}\cite[Lemma 7.1]{handbook}}]\label{stabext}
  If $S \subset L$ is an inclusion of lattices then 
\beq
\widetilde{\opn{O}}(S) \subset \widetilde{\opn{O}}(L)
\eeq
where the extension of $g \in \widetilde{\opn{O}}(S)$ to $\widetilde{\opn{O}}(L)$ is defined by allowing $g$ to act as the identity on $S^{\perp} \subset L$.
\end{lemma}
\subsection{Modular forms}
For a lattice $L$ of signature $(2,n)$ we let $\Omega_L$ denote the Hermitian symmetric space
\beq
\Omega_L = \{[x] \in \PP(L \otimes \CC) \mid (x,x)=0, (x, \overline{x}) > 0 \},
\eeq
where $\overline{x}$ denotes the complex conjugate of $x$.
We use $\dc_L$ to denote the component of $\Omega_L$ fixed by $\opn{O}^+(L)$ and $\dc_L^{\bullet}$ to denote the affine cone of $\dc_L$.
\begin{definition}[{\hspace{1sp}\cite{handbook}}]
  For a subgroup $\Gamma \subset \opn{O}^+(L)$ of finite index,    a \textbf{modular form of weight $k$ and character} $\chi:\Gamma \rightarrow \CC^*$ for $\Gamma$ is a holomorphic function $F: \dc_L^{\bullet} \rightarrow \CC$ such that
  $F(tZ) = t^{-k} F(Z)$ and
  $F(gZ) = \chi(g) F(Z)$
  for all $Z \in \dc_L^{\bullet}$ and all $t \in \CC^*$.
  We use $M_k(\Gamma, \chi)$ to denote the space of weight $k$ modular forms with character $\chi$ and let $M_*(\Gamma, \chi) := \bigoplus_k M_k(\Gamma, \chi)$.
\end{definition}
\subsection{Fourier expansions}
To define Fourier expansions of modular forms, we begin by exhibiting $\dc_L$ as a tube domain, following \cite{handbook}.
A primitive isotropic vector $c \in L$ identifies $\dc_L$ with the affine quadric
\beq
\dc_{L, c} = \{ Z \in \dc_L^{\bullet} \mid (Z, c) = 1 \} \cong \dc_L.
\eeq
If $b \in L^{\vee}$ satisfies $(c,b)=1$ then the signature $(1, n-1)$ lattice $L_c := c^{\perp}/\la c \ra$ can be realised as a sublattice of $L$ by 
\beq
L_c \cong L_{c,b} = L \cap c^{\perp} \cap b^{\perp} 
\eeq
and we define the cone 
\beq
C(L_c) := \{ x \in L_c \otimes \RR \mid (x,x)>0 \}.
\eeq
The two components of $C(L_c)$ are interchanged by elements of negative spinor norm in $\opn{O}(L_c \otimes \RR)$:  we let $C^+(L_c)$ denote the component preserved by $\opn{O}^+(L_c \otimes \RR)$.
The tube domain $\hc_{L,c}$ is given by  
\beq
\hc_{L,c} = L_c \otimes \RR + i C^+(L_c)
\eeq
and we define the isomorphism $\hc_{L,c} \xrightarrow{\sim} \dc_{L, c} \cong \dc_L$ by
\beq
z \mapsto [z] = z \op \left ( b - \frac{(z,z) + (b,b)}{2}c \right ).
\eeq
If $\overline{C}^+(L_c^{\vee})$ denotes the closure of  $C^+(L_c^{\vee})$ then each $F_j \in M_k(\Gamma, \chi)$ admits a Fourier expansion 
\beq
F_j(Z) = \sum_{l \in \overline{C}^+(L^{\vee})} f_j(l) e^{2 \pi i (l, Z)}
\eeq
where $Z \in \hc_{L, c}$ \cite{handbook}.

One often wishes to calculate the Fourier expansion of $F_1 F_2$ to precision $O(e^{2 \pi i(p, Z)})$ for some $p \in L_c^{\vee}$.
As the cones $\overline{C}^+(L_c^{\vee})$ can be large, term-by-term multiplication can be very expensive.
However, as  
\beq
F_1 F_2 =
 \sum_{l_3 \in \overline{C}^+(L_c^{\vee})} \left (  \sum_{\substack{l_1, l_2 \in \overline{C}^+(L_c^{\vee}) \\ l_1 + l_2 = l_3}} f_1(l_1) f_2(l_2) \right )e^{2 \pi i (l_3, Z)} \in M_*(\Gamma, \chi_1 \otimes \chi_2),
\eeq
then, as explained in \cite{RaumImplement}, one only needs to perform the inner multiplication for representatives of the orbits of $\overline{C}^+(L_c^{\vee})$ for a subgroup
\begin{equation}\label{ThetaDef}
\Theta \subset  \opn{stab}_{\Gamma}(C^+(L_c^{\vee})) \cap \opn{Ker}(\chi_1 \otimes \chi_2).
\end{equation}
One can often take $\Theta$ to be a group of the form $\opn{O}^+_{\Ac}(L_c)$ or $\opn{SO}^+_{\Ac}(L_c)$. 
Indeed, under the mild conditions of Theorem \ref{abelianisationthm} and Lemma \ref{thetalemma}, this is always the case. 
\begin{theorem}[\hspace{1sp}\cite{abelianisation}]\label{abelianisationthm}
  Let $L$ be a lattice.
  \begin{enumerate}
\item  Suppose $L$
  represents -2,
  has Witt index $\geq 2$ over $\RR$,
  $\opn{rank}_3(L) \geq 5$ and
  $\opn{rank}_2(L) \geq 6$
  (the \textbf{Kneser conditions}).
  If $L$ contains a single orbit of $-2$ vectors under $\opn{SO}^+(L)$ then $\opn{Hom}(\opn{SO}^+(L), \CC^*)$ is trivial.
\item If $L$ contains at least two copies of the hyperbolic plane and both $\opn{rank}_2(L) \geq 6$, $\opn{rank}_3(L) \geq 5$ then $\opn{Hom}(\opn{SO}^+(L), \CC^*)$ is trivial.
\end{enumerate}
\end{theorem}
\begin{lemma}\label{thetalemma}
  Let $\Gamma \subset \opn{O}(L)$ where $L$ is a lattice of signature $(2,n)$. 
  If $\widetilde{\opn{SO}}^+(L) \subset \Gamma$ and $\opn{Hom}(\opn{SO}^+(L), \CC^*)$ is trivial then one can take $\Theta=\widetilde{\opn{SO}}^+(L_c)$ in \eqref{ThetaDef}. 
\end{lemma}
\begin{proof}
  By Lemma \ref{stabext}, $\widetilde{\opn{SO}}^+(L_c) \subset \widetilde{\opn{SO}}^+(L)$ and the result follows as $\widetilde{\opn{SO}}^+(L_c) \subset \opn{Stab}_{\Gamma}(\overline{C}^+(L_c))$. 
\end{proof}
\subsection{The Baily-Borel compactification}
If $L$ is a lattice of signature $(2, n)$ and $\Gamma \subset \opn{O}^+(L)$ is a subgroup of finite index, the quotient
\beq
\fc_L(\Gamma) = \dc_L/ \Gamma
\eeq
is known as an \emph{orthogonal modular variety}.
Orthogonal modular varieties are quasi-projective \cite{bailyborel} but, in general, are non-compact.
The simplest compactification of $\fc_L(\Gamma)$ is the \emph{Baily-Borel compactification} $\fc_L(\Gamma)^*$, which can be defined as $\opn{Proj} M_*(\Gamma, \mathds{1})$ \cite{bailyborel}.
\begin{theorem}[\hspace{1sp}\cite{handbook}]\label{BailyBorelThm}
  The Baily-Borel compactification $\fc_L(\Gamma)^*$ decomposes  as
  \beq
  \fc_L(\Gamma)^* = \fc_L(\Gamma) \sqcup \bigsqcup_{\Pi} C_{\Pi} \sqcup \bigsqcup_{\ell} Q_{\ell} 
  \eeq
where each $C_{\Pi}$ is a modular curve, each $Q_{\ell}$ is a point and the indices $\Pi$ and $\ell$ are taken over representatives of $\Gamma$-orbits of totally isotropic planes and isotropic lines in $L \otimes \QQ$, respectively.
A point $Q_{\ell}$ is contained in the closure of $C_{\Pi}$ if and only if representatives can be chosen such that $\ell \subset \Pi$.
\end{theorem}
In \S\ref{BoundarySection}-\ref{ExampleSection} we explain how to calculate the boundary configuration of $\fc_L(\Gamma)^*$ using
Algorithms \ref{TitsAlgo}, \ref{MaximalAlgo} and \ref{SuperGroupAlg}.
 \section{Orbits of vectors}
For a lattice $L$, let $\Gamma \subset \opn{O}(L)$ be a subgroup and let $v_1, v_2 \in L^{\vee} \subset L \otimes \QQ$.
If there exists $g \in \Gamma \subset \opn{O}(L)$ such that $gv_1 = v_2$ we say that $v_1$ and $v_2$ are \emph{equivalent under $\Gamma$}, which we denote by  $v_1 \sim_{\Gamma} v_2$; otherwise,  $v_1$ and $v_2$ are said to be \emph{inequivalent under $\Gamma$}, which we denote by $v_1 \not \sim_{\Gamma} v_2$. 
If $v_1$ and $v_2$ are non-isotropic, these relations can be calculated using Algorithms \ref{non_iso_def_alg}, \ref{non_iso_indef_alg0} and \ref{non_iso_indef_alg}.
We will treat the isotropic case separately in \S\ref{BoundarySection}.
In order to prove Algorithm \ref{non_iso_def_alg}, \ref{non_iso_indef_alg0} and \ref{non_iso_indef_alg} we shall need Lemma \ref{extlem}, which is contained but not proved in \cite{Nikulin}.
\begin{lemma}\label{extlem}
  For $i=1,2$, let $S_i \subset L$ be a sublattice and let $K_i:=S_i^{\perp} \subset L$.
  The inclusion 
  \begin{equation}\label{lattice_inclusion}
  S_i \op K_i \subset L \subset L^{\vee} \subset S_i^{\vee} \op K_i^{\vee}
  \end{equation}
  defines subgroups
  \beq
  H_i = L/(S_i \op K_i) \subset D(S_i) \op D(K_i)
  \eeq
  and natural homomorphisms $p_{S_i}:H_i \rightarrow D(S_i)$ and $p_{K_i}:H_i \rightarrow D(K_i)$.
  We let $H_{S_i}$ and $H_{K_i}$ denote the respective images of $p_{S_i}$ and $p_{K_i}$ and define $\gamma_i=p_{K_i} \circ p_{S_i}^{-1}:H_{S_i} \rightarrow H_{K_i}$.
  Then, 
  \begin{enumerate}
  \item the homomorphisms $p_{S_i}$ and $p_{K_i}$ are monomorphisms if and only if both $S_i \subset L$ and $K_i \subset L$ are primitive.
  \item If $\varphi:S_1 \xrightarrow{\sim} S_2$ then $\varphi$ extends to an element of $g \in \opn{O}(L)$ if and only if there exists $\psi:K_1 \rightarrow K_2$ such that $\overline{\psi} \circ \gamma_1 = \gamma_2 \circ \overline{\varphi}$.
The map $g$ is given by the natural extension of $\varphi \circ \psi$ from $S \op K$ to $L$.
  \item Let $\iota_i:D(L) \hookrightarrow (S_i^{\vee} \op K_i^{\vee})/L$ be the natural map defined by \eqref{lattice_inclusion}.
    If $g \in \opn{O}(L)$ is the extension of $\varphi \op \psi$, then $g \in \opn{O}_{\Ac}(L)$ if and only if $\iota_2^{-1} \circ (\overline{\varphi} \op \overline{\psi} ) \circ \iota_1 \in \Ac$, where $\Ac \subset \opn{O}(D(L))$.
  \end{enumerate}
\end{lemma}
\begin{proof}
  \begin{enumerate}
  \item
    Let $S:=S_i$ and $K:=K_i$. 
    For $x \in L$, let $s \in S^{\vee}$ and $k \in K^{\vee}$ be such that $x = s + k$.
Suppose $S \subset L$ is primitive and $0 \not \equiv x \bmod{S \op K}$.
    As $S$ is primitive, if $p_S(x) = 0$ then $s \in S$ and $k \in L \cap K^{\vee} = K$ (as $K=S^{\perp}$ is primitive).
    Hence, from the contradiction $0 \equiv x \bmod{S \op K}$, $p_S$ is a monomorphism.
Now suppose $p_S$ and $p_K$ are monomorphisms.
    If $x \bmod{ S}$ is torsion then $k=0$ and as  $p_S$ is a monomorphism, the contradiction $0 \equiv p_S(x) \equiv x \bmod{S}$ implies $S$ is primitive.
    The argument for $K$ is identical.
  \item  The map $\varphi \op \psi$ extends to $g \in \opn{O}(L)$ if and only if it preserves $H$ which,  by definition of $\gamma_1$ and $\gamma_2$, is equivalent to $\overline{\psi} \circ \gamma_1 = \gamma_2 \circ \overline{\varphi}$.
  \item Immediate from definition.
  \end{enumerate}
\end{proof}
\subsection{Orbits of non-isotropic vectors}
\begin{algo}\label{non_iso_def_alg}
  For a lattice $L$ of rank $n$,  let $\Gamma \subset \opn{O}(L)$ be a subgroup and let $v_1, v_2 \in L \otimes \QQ$.
  If $v_1$ is non-isotropic and $v_1^{\perp}$ is definite then one can determine if $v_1  \sim_{\Gamma} v_2$ by proceeding as follows. 
\end{algo}
\begin{enumerate}
\item \textbf{For} $i \in \{1,2 \}$, \textbf{let} $c_i \in \QQ_{>0}$ be minimal such that $c_i v_i \in L$. \label{ANID1}
\item    \textbf{If} $v_1^2 \neq v_2^2$ \textbf{or} $c_1 \neq c_2$ \textbf{return} $v_1 \not \sim_{\Gamma} v_2$.    \label{ANID2}
\item \textbf{For} $i \in \{1,2 \}$         \label{ANID3}
  \begin{enumerate}
  \item \textbf{Let} $w_i := c_i v_i$.
  \item \textbf{Let} $(q_1 \vert \ldots \vert q_n):=Q(\hat{w_i})$.
  \item \textbf{Let} $K_i=\la k_{ij} \mid j=1, \ldots, n-1 \ra$ where $k_{ij}=q_{j+1}$.
  \item \textbf{Let} $\iota_i:=(w_i \vert k_{i1} \vert \ldots \vert k_{i(n-1)})$.
  \end{enumerate}
\item  \textbf{Let} $\varphi$ be the map $w_1 \mapsto w_2$. \label{ANID4}
\item \textbf{For} $\psi \in \opn{Iso}(K_1, K_2)$ \label{ANID5}
  \begin{enumerate}
  \item \textbf{Let} $\theta:=\iota_2 \circ (\varphi \op \psi) \circ \iota_1^{-1}$
  \item \textbf{If} { $\theta \in \Gamma$}
    \textbf{return} $v_1 \sim_{\Gamma} v_2$.
  \end{enumerate}
\item \textbf{Return} $v_1 \not \sim_{\Gamma} v_2$. \label{ANID6}
\end{enumerate}
\begin{proof}
  As $v_i^2$ and $c_i$ are invariant under $\opn{O}(L)$, \ref{ANID1} and \ref{ANID2} serve as preliminary tests.
In \ref{ANID3}, we calculate $K_i:=w_i^{\perp} \subset L$.
  By the Smith normal form (as $d_2 = \ldots = d_n = 0$), $K_i \subset L$ is primitive and as $K_i \perp w_i$ then $K_i = w_i^{\perp} \subset L$.
In \ref{ANID4}, we define an embedding $\iota_i$ of $\la w_i \ra \op K_i \subset L$.
By Lemma \ref{extlem}, $w_1 \sim_{\Gamma} w_2$ if and only if there exists $\varphi:=g \vert_{\la w_1 \ra}: w_1 \mapsto w_2$
and
$\psi:=g \vert_{K_1}: K_1 \rightarrow K_2$ extending to $\Gamma$.
We search for suitable $\psi$ in \ref{ANID5}.
\end{proof}
As $K_1$ is definite, $\opn{Iso}(K_1, K_2)$ in \ref{ANID5} can be calculated efficiently using the fast isomorphism testing of \cite{PleskenPohst, PleskenSouvignier}.
In \ref{ANID5}(b), one has to verify whether $\theta \in \Gamma$.
If $\Gamma=\opn{SO}_{\Ac}(L)$ or $\opn{O}_{\Ac}(L)$ one can simply check if $\theta \in \opn{GL}(N, \ZZ)$ and $\overline{\theta} \in \Ac$.
However, if $\Gamma =\opn{SO}^+_{\Ac}(L)$ or $\opn{O}^+_{\Ac}(L)$ one also has to verify a spinor norm condition.
Effective methods exist (e.g. \cite[p.18-20]{cassels}) for decomposing $\psi$ into a product of reflections, from which one can calculate  $\opn{sn}_{\RR}(\theta)$ using \eqref{spinordef}; however, in the case of Lorenztian $L$, we demonstrate an alternative approach in Example \ref{nonisoorbit}.
\begin{example}\label{nonisoorbit}
  Let $L=U \op A_3$ where
  \begin{center}
    \begin{tabular}{ccc}
    $G(U) = \bpm 0 & 1 \\ 1 & 0 \epm$
    & and &
      $G(A_3) = - \bpm 2 & 1 & 0 \\ 1 & 2 & 1 \\ 0 & 1 & 2 \epm$
    \end{tabular}
    \end{center}
  and suppose $v_1, v_2 \in L \otimes \QQ$
  are given by
  $v_1= (4,4,1,2,-1)$ and $v_2 = (36, 144, 5, -30, 83)$.
  If $\Gamma = \widetilde{\opn{O}}^+(L)$ then $v_1 \sim_{\Gamma} v_2$.
\end{example}
\begin{proof}
  We apply Algorithm \ref{non_iso_def_alg}.
  In \ref{ANID1} - \ref{ANID3}, we have  $c_1=c_2=1$, $v_1^2=v_2^2=20$, $w_1=v_1$, $w_2 = v_2$, 
  \begin{center}
    \begin{tabular}{ccc}
      $Q(\hat{w_1}) =
      \bpm
      0 & 1 & 0 & 0 & 0 \\
      0 & 0 & 1 & 0 & 0 \\
      0 & 0 & 0 & 1 & 0 \\
      -1 & 1 & 1 & -1 & 0 \\
      0 & 0 & 0 & 0 & 1
      \epm$,
      &  &
      $Q(\hat{w_2}) =
      \bpm
      0 & 1 & 0 & 0 & 0 \\
      0 & 0 & 1 & 0 & 0 \\
      0 & 0 & 0 & 1 & 0 \\
      -5 & 180 & 45 & 25 & 34 \\
      1 & -36 & -9 & -5 & -7
      \epm$, \\
      \\
      $\iota_1 =
      \bpm
      4 & 1 & 0 & 0 & 0 \\
      4 & 0 & 1 & 0 & 0 \\
      1 & 0 & 0 & 1 & 0 \\
      2 & 1 & 1 & -1 & 0 \\
      -1 & 0 & 0 & 0 & 1
      \epm$
      & and &
      $\iota_2 =
      \bpm
      36 & 1 & 0 & 0 & 0 \\
      144 & 0 & 1 & 0 & 0 \\
      5 & 0 & 0 & 1 & 0 \\
      -30 & 180 & 45 & 25 & 34 \\
      83 & -36 & -9 & -5 & -7
      \epm$.      
    \end{tabular}
  \end{center}
  The Gram matrices of $K_1$ and $K_2$ are given by
  \begin{center}
    \begin{tabular}{ccc}
      $G(K_1) =
      \bpm
      -2 & -1 & 1 & -1 \\
      -1 & -2 & 1 & -1 \\
      1 & 1 & -2 & 1 \\
      -1 & -1 & 1 & -2
      \epm$
      & and &
      $G(K_2) =
      \bpm
      -54432 & -13607 & -7740 & -10260 \\
      -13607 & -3402 & -1935 & -2565 \\
      -7740 & -1935 & -1102 & -1459 \\
      -10260 & -2565 & -1459 & -1934
      \epm$.
    \end{tabular}
  \end{center}
  In \ref{ANID5}, we search over $\opn{Iso}(K_1, K_2)$ (e.g. using \cite{FinckePohst, PleskenPohst, PleskenSouvignier}) for $\theta=\iota_2 \circ (\varphi \op \psi) \circ \iota_1^{-1}  \in \Gamma$.
  We find $\theta \in \opn{O}(L)$, given by
  \begin{center}
    \begin{tabular}{ccc}
        $\theta =
        \bpm
        11 & 5 & -11 & -13 & -9 \\
        43 & 21 & -46 & -51 & -36 \\
        1 & 1 & -1 & -2 & -2 \\
        -9 & -5 & 10 & 12 & 8 \\
        25 & 12 & -26 & -30 & -21
        \epm
        $
        & for &
        $\psi =
        \bpm
        -2 & -8 & 2 & -9 \\
        -8 & -30 & 5 & -36 \\
        -1 & -1 & 1 & -2 \\
        22 & 83 & -18 & 97
        \epm$.
    \end{tabular}
    \end{center}
The vector $w = \frac{1}{4}(0,0,3,-2,1) \in L^{\vee} \subset L \otimes \QQ$ represents a generator for the cyclic group $D(L)$.
  Therefore, to determine $\theta \in \widetilde{\opn{O}}^+(L)$, we verify $w \equiv \theta w \bmod{L}$.  
As $L$ is Lorentzian, $ g \in \opn{O}^+(L)$  if and only if $ g (\Cc(L)^+) = \Cc(L)^+$. 
  The quadratic form of $L$ is diagonalised to $\opn{diag}(-2, -1, 1, -\sqrt{2} - 2, \sqrt{2} - 2)$ by
  \beq
  P =
  \bpm
  0 & -1 & 1 & 0 & 0 \\
  0 & 1 & 1 & 0 & 0 \\
  -1 & 0 & 0 & 1 & 1 \\
  0 & 0 & 0 & \sqrt{2} & -\sqrt{2} \\
  1 & 0 & 0 & 1 & 1
  \epm
  \eeq
 and so, in new coordinates given by $P$, 
  \beq
  \Cc(L)^+ = \{(x_0, x_1, x_2, x_3, x_4) \in \RR^5 \mid \text{$x_0>0$ and $\omega_0 x_0^2 > \sum_{i=1}^4 \omega_i x_i^2$} \}. 
  \eeq
  As ${}^{\tau}x=(1,0,0,0,0) \in \Cc(L)^+$ and $(P^{-1} \theta P)^{\tau}x = ^{\tau} (3,\,4,\,6,\,-\frac{1}{2} \, \sqrt{2} + 1,\,\frac{1}{2} \, \sqrt{2} + 1)$
  then $\theta$ preserves $\CC(L)^+$, implying $v_1 \sim_{\Gamma} v_2$.
\end{proof}
\begin{algo}\label{non_iso_indef_alg0}
  If $L$ is a lattice of rank $n$, let $v_1, v_2 \in L \otimes \QQ$ and $\Gamma:=\opn{O}_{\Ac}(L)$ where $\Ac \subset \opn{O}(D(L))$ is subgroup. 
  If, simultaneously, 
$v_1$ is non-isotropic,
  $v_1^{\perp}$ is indefinite, 
  the natural map $\opn{O}(L) \rightarrow \opn{O}(D(L))$ is surjective,
then one can determine if $v_1 \sim_{\Gamma} v_2$ as follows.
  \begin{enumerate}
  \item \textbf{For} $i \in \{1,2\}$ \label{NIIA01}
    \begin{enumerate}
    \item \textbf{Let} $c_i \in \QQ_{>0}$ be minimal such that $w_i:=c_i v_i \in L$.
    \item \textbf{If} $c_1 \neq c_2$ \textbf{or} $v_1^2 \neq v_2^2$ \textbf{then return} $v_1 \not \sim_{\Gamma} v_2$. 
    \item \textbf{Let} $K_i:=w_i^{\perp} \subset L$. 
    \item For the natural inclusion
      \beq
      \la w_i \ra \op K_i
      \subset L
      \subset L^{\vee}
      \subset \la w_i \ra^{\vee} \op K_i^{\vee},
      \eeq
      \textbf{let} $H_i:=L/\la w_i \ra \op K_i \subset D(\la w_i \ra ) \op D(K_i)$. 
    \item \textbf{Let} $\iota_i$ be the map $\iota_i:D(L) \xrightarrow{\sim} D(\la w_i \ra) \op D(K_i) \bmod{H_i}$. 
    \end{enumerate}
  \item \textbf{If} $K_1 \not \cong K_2$ \textbf{then return} $v_1 \not \sim_{\Gamma} v_2$.
  \item \textbf{For} $\overline{\varphi} \op \overline{\psi} \in \{ \pm 1 \} \op \opn{Iso}(q_{K_1}, q_{K_2})$ \label{NIIA02}
    \bi
  \item[] \textbf{If}
    $(\overline{\varphi} \op \overline{\psi})(H_1) = H_2$ \textbf{and} $\iota_2^{-1} \circ (\overline{\varphi} \op \overline{\psi}) \circ \iota_i \in \Ac$ \textbf{then}
    \bi
  \item[] \textbf{return } $v_1 \sim_{\Gamma} v_2$. 
    \ei
    \ei
  \item \textbf{Return} $v_1 \not \sim_{\Gamma} v_2$. \label{NIIA03}
  \end{enumerate}
\end{algo}
\begin{proof}
  Lemma \ref{extlem} with \ref{NIIA01}(a) and \ref{NIIA01}(b) serving as preliminary tests.
\end{proof}
Algorithm \ref{non_iso_indef_alg0} can be rephrased in coordinate form as follows.
\begin{algo}\label{non_iso_indef_alg}
  If $L$ is a lattice of rank $n$,
  let $v_1, v_2 \in L \otimes \QQ$  and
  $\Gamma:=\opn{O}_{\Ac}(L)$ where $\Ac \subset \opn{O}(D(L))$ is a subgroup.
  If, simultaneously, 
$v_1$ and $v_2$ are non-isotropic,
  $v_1^{\perp}$ is indefinite, 
  the natural map $\opn{O}(L) \rightarrow \opn{O}(D(L))$ is surjective, 
then one can determine if $v_1 \sim_{\Gamma} v_2$ as follows. 
  \begin{enumerate}
  \item \textbf{For} $i \in \{1,2\}$  \label{NIIA1}
    \begin{enumerate}
    \item \textbf{Let} $c_i \in \QQ_{>0}$ be minimal such that $w_i:=c_i v_i \in L$.
    \item \textbf{Let} $\alpha_i := w_i^2 / \vert w_i^2 \vert$.
    \end{enumerate}
  \item \textbf{If} $v_1 ^2 \neq v_2^2$ \textbf{or} $c_1 \neq c_2$ \textbf{return} $v_1 \not \sim_{\Gamma} v_2$. \label{NIIA2}
  \item \textbf{For} $i \in \{1,2\}$ \label{NIIA3} 
    \begin{enumerate}
\item \textbf{Let} $(q_1 \vert \ldots \vert q_n ):=Q(\hat{w_i})$.
    \item \textbf{Let} $K_i:= \la k_{ij} \vert j=1, \ldots, n-1 \ra$ where $k_{ij}=q_{j+1}$.
    \item \textbf{Let}
      $[d_1, \ldots, d_n]_{n,n}:=P(G(K_i)) G(K_i) Q(G(K_i))$ and identify
\begin{equation}\label{DKcanonical}
D(K_i) \cong \bigoplus_j C_{d_j}.
\end{equation}
    \item \textbf{Let}
\beq
      f_{il} := \frac{1}{d_k} \sum_{j=1}^n q_{jl} k_{il}  
      \eeq
for $l=1, \ldots, n-1$.
    \end{enumerate} 
  \item \textbf{If} $K_1 \not \cong K_2$ \textbf{then} \label{NIIA4}
    \textbf{return} $v_1 \not \sim_{\Gamma} v_2$ \textbf{else} 
    \begin{enumerate}
    \item \textbf{Let}
      $\underline{e}_1 = {}^{\tau}(1, 0, \ldots, 0)$, 
      $\theta_{i1} := (w_i \vert k_{i1} \vert \ldots \vert k_{in} )$, 
      $\theta_{i2} := (w_i^2) \op G(K_i)$, 
      $\theta_{i3} := (\alpha_i \underline{e}_1 \vert f_{i1} \vert \ldots \vert f_{i(n-1)})$ and
      $\lambda_i:= \theta_{i3} \circ \theta_{i2} \circ \theta_{i1}^{-1}$.
    \item \textbf{Let} $H_i$ be the subgroup of $D(\la w_i \ra) \op D(K_i)$ generated by the columns of $\lambda_i$ taken modulo ${}^{\tau}(\vert w_i^2 \vert, d_1, \ldots, d_n)$. 
  \item \textbf{If} $H_1 \not \cong H_2$ \textbf{then return} $v_1 \not \sim_{\Gamma} v_2$.
  \item \textbf{Let} $\iota_i:= \lambda_i  \circ G(L)^{-1}$.
    \end{enumerate}
 \item \textbf{For} $\overline{\varphi} \op \overline{\psi} \in \{ \pm 1 \} \op \opn{Iso}(q_{K_1}, q_{K_2})$ \label{NIIA5}
    \bi
    \item[] \textbf{If} $(\overline{\varphi} \op \overline{\psi})(H_1) = H_2$ \textbf{and} $\theta:=\iota_2^{-1} \circ (\overline{\varphi} \op \overline{\psi}) \circ \iota_1 \bmod{L} \in \Ac$ \textbf{then}
      \bi
      \item[] \textbf{return} $v_1 \sim_{\Gamma} v_2$.  
      \ei
      \ei
    \item \textbf{Return} $v_1 \not \sim_{\Gamma} v_2$.
  \end{enumerate}
\end{algo}
\begin{proof}
The algorithm is essentially a rephrasing of Algorithm \ref{non_iso_indef_alg0}.
The term $\alpha_i$ in \ref{NIIA1} corresponds to the Smith normal form of $\la w_i \ra$, which we will use in \ref{NIIA4}.
  In \ref{NIIA3}(a), we calculate bases for the lattices $K_i$ as in Algorithm \ref{non_iso_def_alg} and in \ref{NIIA3}(b), we calculate the groups $D(K_i)$.
  Step \ref{NIIA3}(c) calculates representatives $f_{il}$ in $K_i^{\vee} \subset K_i \otimes \QQ$ for the canonical basis of $D(K_i)$ in \eqref{DKcanonical}.
  In \ref{NIIA4}, we calculate maps
\begin{align*}
    \iota_i:D(L) \rightarrow D(\la w_i \ra) \op D(K_i)
      && \text{and} &&
    \lambda_i:L / \la w_i \ra \op K_i \rightarrow D( w_i) \op D(K_i) 
  \end{align*}
  defined on generators in $L^{\vee}$, $L$ and $D(\la w_i) \op D(K_i)$ with $\lambda_i$ and $\iota_i$ as in \ref{NIIA4}(a) and \ref{NIIA4}(d), respectively.
  We conclude in \ref{NIIA5} by verifying the existence of
  $\varphi \op \psi:\la w_1 \ra \op K_1 \rightarrow \la w_2 \ra \op K_2$ extending to $\Gamma$.
\end{proof}
Unlike in Algorithm \ref{non_iso_def_alg}, the lattices $K_i$ of Algorithm \ref{non_iso_indef_alg0} and \ref{non_iso_indef_alg} are indefinite, precluding an application of the isomorphism tests of \cite{PleskenPohst, PleskenSouvignier}.
However, one can typically determine if $K_1 \cong K_2$ by using Theorem \ref{singclassthm} (which is originally due to \cite{KneserIndef} and rephrased in the language of discriminant forms in \cite{Nikulin}).
\begin{theorem}[{\hspace{1sp}\cite{KneserIndef}\cite[Theorem 1.13.2/1.14.2]{Nikulin}}]\label{singclassthm}
  Let $q$ be the discriminant form of a lattice.
  If both
\begin{enumerate}
\item $t_+ \geq 1$, $t_- \geq 1$ and $t_+ + t_- \geq 3$;
\item  $t_+ + t_- \geq 2 + l(q)$ where $l(q)$ is the minimum number of generators for the underlying group of $q$, 
\end{enumerate}
then there exists a lattice $L$ of signature $(t_+, t_-)$ with $q_L = q$. Furthermore, the natural map $\opn{O}(L) \rightarrow \opn{O}(D(L))$ is surjective and $\opn{gen}(L)$ contains a single class.
\end{theorem}
Under the conditions of Lemma \ref{2reflem}, Algorithm \ref{non_iso_indef_alg} can be used to determine the equivalence of vectors under $\opn{SO}_{\Ac}^+(L)$.
\begin{lemma}\label{2reflem}
  In the notation of Algorithm \ref{non_iso_indef_alg}, if $K_1$ represents both $\pm 2$ then $v_1 \sim_{\opn{O}_{\Ac}(L)} v_2$ if and only if $v_1 \sim_{\opn{SO}_{\Ac}^+(L)} v_2$. 
\end{lemma}
\begin{proof}
  Suppose $v \in K_1$ satisfies $v^2 = \pm 2$.
  By definition, $\opn{sn}_{\RR}(\sigma_v) = \mp 1$ and, by \cite[Proposition 3.1]{GHSK3},  $\sigma_v \in \widetilde{\opn{O}}(K_1)$.
  If necessary, as $\opn{det}(\sigma_v) = -1$, one can replace $\psi$ in Algorithm \ref{non_iso_indef_alg} by $\psi \circ \sigma_v$ so that $\varphi \op \psi$ in Algorithm \ref{non_iso_indef_alg} has the required determinant and spinor norm. 
\end{proof}
Theorem \ref{singclassthm} and Lemma \ref{indefrep2} can often be used to determine if a lattice represents $\pm 2$.
\begin{lemma}\label{indefrep2}
  Let $K$ be an indefinite lattice with discriminant form $q_K$ and signature $(t_+, t_-)$.
  If $S := \la \pm 2 \ra$, suppose $\delta$ is given by one of
\begin{enumerate}
\item $\delta= q_S \op (-q_K)$; 
\item $\delta=( (q_S \op (-q_K)) \mid \Gamma_{\gamma}^{\perp}) /\Gamma_{\gamma}$, where $\Gamma_{\gamma}$ is the pushout for an inclusion of subgroups $\gamma:q_S \rightarrow q_L$ compatible with the forms $q_S$ and $q_L$.
\end{enumerate}
If $K$ is unique in its genus and there exists a lattice of signature $(t_+, t_-)$ with discriminant form $-\delta$ then $S \subset K$.
\end{lemma}
\begin{proof}
Immediate from Proposition 1.15.1 of \cite{Nikulin} and Theorem \ref{singclassthm}.
\end{proof}
\begin{example}\label{alg2ex}
  Let $L = U \op A_3$ and suppose $v_1, v_2 \in  L \otimes \QQ$ are given by $v_1 = (1,-1,0,0,0)$ and $v_2 = (1,0,1,0,0)$ where $G(U)$ and $G(A_3)$ are as in Example \ref{nonisoorbit}.
  If $\Gamma = \widetilde{\opn{SO}}^+(L)$ then $v_1 \sim_{\Gamma} v_2$.
\end{example}
\begin{proof}
  We apply Algorithm \ref{non_iso_indef_alg}.
  As $L$ is of signature $(2,5)$ and $v_1^2 > 0 $ then $v_1^{\perp}$ is indefinite. 
  By Theorem \ref{singclassthm}, $L$ is unique in its genus and $\opn{O}(L) \rightarrow \opn{O}(D(L))$ is surjective.
  For $i=1,2$, we have $w_i^2=-2$, $w_i=v_i$, $c_i=1$ and $\alpha_i=1$.
  Bases for $K_i$ are given by
  \begin{align*}
    K_1 =
    \bpm
    1 & 0 & 0 & 0 \\
    1 & 0 & 0 & 0 \\
    0 & 1 & 0 & 0 \\
    0 & 0 & 1 & 0 \\
    0 & 0 & 0 & 1 
    \epm
    && \text{and} &&
    K_2 =
    \bpm
    1 & 0 & 0 & 0 \\
    0 & 1 & 0 & 0 \\
    0 & 0 & 1 & 0 \\
    0 & 1 & -2 & 0 \\
    0 & 0 & 0 & 1 
    \epm,
  \end{align*}
  with Gram matrices
  \begin{align*}
    G(K_1) =
    \bpm
    2 & 0 & 0 & 0 \\
    0 & -2 & -1 & 0 \\
    0 & -1 & -2 & -1 \\
    0 & 0 & -1 & -2
    \epm
    && \text{and} &&
    G(K_2) =
    \bpm
    0 & 1 & 0 & 0 \\
    1 & -2 & 3 & -1 \\
    0 & 3 & -6 & 2 \\
    0 & -1 & 2 & -2
    \epm.
  \end{align*}
  We calculate $D(K_1) \cong C_1 \op C_1 \op C_2 \op C_4 \cong D(K_2)$, implying $K_1 \cong K_2$ by Theorem \ref{singclassthm}.
  As
  \begin{align*}
    Q(G(K_1)) =
    \bpm
    0 & -1 & 3 & 2 \\
    -1 & 0 & 2 & 1 \\
    1 & 0 & -2 & -2 \\
    -1 & -1 & 4 & 3
    \epm
    && \text{and} &&
    Q(G(K_2)) =
    \bpm
    1 & 1 & 1 & 0 \\
    1 & 0 & 0 & 0 \\
    0 & 0 & -1 & -1 \\
    -1 & 0 & -2 & -3
    \epm,
  \end{align*}
  if $x_1, x_2, y_1, y_2 \in L \otimes \QQ$ are defined by  
  $x_1 := \frac{1}{2}(3,3,2,-2,4)$,
  $x_2 := \frac{1}{4}(2,2,1,-2,3)$,
  $y_1 := \frac{1}{2}(1,0,-1,2,-2)$ and
  $y_2 := \frac{1}{4}(0,0,-1,2,-3)$ then $D(K_1) \cong \la x_1, x_2 \ra \bmod{L}$ and $D(K_2) \cong \la y_1, y_2 \ra \bmod{L}$ (we ignore the $C_1$ terms for brevity).
  Therefore, 
  \begin{align*}
    q_{K_1}(a,b) = -\frac{3a^2}{2}  - \frac{b^2}{4} - ab \bmod{2 \ZZ}
    && \text{and} &&
    q_{K_2}(a,b) = -\frac{3a^2}{2} - \frac{3b^2}{4} \bmod{2 \ZZ}
  \end{align*}
  where $(a,b) \in C_2 \op C_4$.
  As 
  \begin{align*}
    \lambda_1
    =
    \bpm
    1 & -1 & 0 & 0 & 0 \\
    1 & 1 & 0 & -1 & -2 \\
    1 & 1 & 1 & -2 & -3 \\
    1 & 1 & 2 & 0 & -2 \\
    0 & 0 & 4 & -4 & 0
    \epm
    && \text{and} &&
    \lambda_2
    =
    \bpm
    0 & -1 & 2 & 1 & 0 \\
    0 & 1 & 0 & 0 & 0 \\
    1 & 0 & -1 & -2 & -1 \\
    0 & -1 & 0 & 3 & 2 \\
    0 & 0 & 0 & 4 & -4
    \epm 
  \end{align*}
  then $H_1 \cong H_2 \cong \la (1,0) \ra \subset C_2 \op C_4 \cong D(K_1) \cong D(K_2)$.
  We verify condition \ref{NIIA5} by considering $P(G(L)) \theta P(G(L))^{-1}$ where
  \beq
  P(G(L)) =
  \bpm
  1 & 0 & 0 & 0 & 0 \\
  0 & 1 & 0 & 0 & 0 \\
  0 & 0 & 1 & 0 & 0 \\
  0 & 0 & 0 & 1 & 0 \\
  0 & 0 & -1 & -2 & 1
  \epm.
  \eeq
We find that if $\overline{\varphi}=(1)$ and
  \beq
  \overline{\psi} =
  \bpm
  1 & 0 & 0 & 0 \\
  0 & 1 & 0 & 0 \\
  0 & 0 & 1 & 1 \\
  0 & 0 & 0 & 3
  \epm
  \eeq
  then $\theta$ acts trivially on $D(L)$.
Therefore, $v_1 \sim_{\opn{O}_{\Ac}(L)} v_2$ and, by Lemma \ref{2reflem}, $v_1 \sim_{\Gamma} v_2$.
\end{proof}
 \section{Algorithms for Tits' buildings and isotropic vectors}\label{BoundarySection}
We now explain how to calculate Tits' buildings and orbits of isotropic vectors  for a group $G_1 \subset \opn{O}^+(L \otimes \QQ)$ where $L$ is a lattice of signature $(2,n)$.
We will assume throughout that $\widetilde{\opn{SO}}^+(L) \subset G_1$.
\begin{defn}\label{TitsBuildingDefn}
  Let $L$ be a lattice of signature $(2,n)$.
  For a group $G \subset \opn{O}^+(L \otimes \QQ)$, 
  let $\pc$ and $\cc$ denote the $G$-orbits of totally isotropic subspaces of dimension 1 and 2 in $L \otimes \QQ$, respectively.
  Then the \textbf{Tits' building} $\bc(G) = (\nc, \ec) $ of $G$ is the bipartite graph with node set $\nc:=\pc \sqcup \cc$ and edge set $\ec$, where an edge is drawn between $[\Pi] \in \cc$ and $[\ell] \in \pc$ if and only if $g \ell \in \Pi$ for some $g \in G$.
\end{defn}
When drawing Tits' buildings, we adopt the convention that nodes in $\pc$ will be coloured black and nodes in $\Cc$ will be coloured white.
Nodes may or may not be numbered.
We will typically use $\ell$ to denote isotropic lines and $\Pi$ to denote isotropic planes.
We note that Definition \ref{TitsBuildingDefn} has an equivalent formulation with `primitive totally isotropic sublattice of $L$' replacing `totally isotropic subspace of $L \otimes \QQ$'.
\begin{algo}\label{TitsAlgo}
  Let $G_1 \subset G_2 \subset \opn{O}^+(L \otimes \QQ)$ be an inclusion of groups such that $\vert G_2 : G_1 \vert < \infty$.
  Suppose $\bc(G_2) = (\nc_2, \ec_2)$ where $\nc_2 = \pc_2 \sqcup \cc_2$, $\pc_2 = \{[\ell_i] \}$ and $\cc_2 = \{ [\Pi_i] \}$.
  Then the Tits' building $\bc(\nc_1, \ec_1)$ (where $\nc_1 = \pc_1 \sqcup \cc_1$) can be calculated from $\bc(G_2)$ as follows.
  \begin{enumerate}
  \item \textbf{Let}
    $\ec_1 := \emptyset$,
    $\pc_1 := \emptyset$,
    $\cc_1 := \emptyset$. \label{init}
  \item \textbf{Let} $\gc:=\{g_i\}$ be a transversal for $G_2/G_1$.
  \item \label{nodes} \textbf{For} $[E] \in \nc_2$
    \begin{enumerate}
    \item \textbf{Let} $\hc$ be a set of representatives for $\gc/\sim$ where, if $\lc_j$ is a transversal for
\beq
\opn{Stab}_{G_1}(q_jE) \backslash \opn{Stab}_{G_2}(q_j E),
\eeq
then $g_i \sim g_j$ if and only if $l_k g_j g_i^{-1} \in G_1$ for $l_k \in \lc_j$.
    \item \textbf{If} $E \in \pc_2$
     \textbf{let} $\pc_1:=\pc_1 \cup \hc.[E]$
     \textbf{else}
     \textbf{let} $\cc_1 := \cc_1 \cup \hc.[E]$.
    \end{enumerate}
  \item \label{edges} \textbf{For} $([\ell], [\Pi]) \in \pc_1 \times \cc_1$
    \begin{enumerate}
    \item \textbf{Let} $X$ be a finite set of representatives for all lines in $\Pi$ up to equivalence in $\opn{Stab}_{G_2}(\Pi)$.
    \item \textbf{Let} $\rho: \opn{Stab}_{G_2}(\Pi) \rightarrow \opn{GL}(\Pi)$ be the restriction homomorphism;
      \textbf{let} $\Gamma \subset \opn{Stab}_{G_1}(\Pi)$ be such that $\vert \rho(\Gamma) : \opn{Im}(\rho) \vert < \infty$;
      and \textbf{let} $\jc \subset \opn{Im}(\rho)$ be a finite set containing representatives for all cosets $\rho(\Gamma) \backslash \opn{Im}(\rho)$.
    \item \textbf{Let} $Y := \jc .X$.
    \item \textbf{Let} $\uc$ be a transversal for $\opn{Stab}_{G_1}(\ell) \backslash \opn{Stab}_{G_2}(\ell)$.
    \item \textbf{If}, for some $y \in Y$, both
      \begin{enumerate}
      \item there exists $\tau(y, \ell) \in G_1$ such that $\tau(y, \ell)y = \ell$
      \item $u \tau(y, \ell) \in G_1$ for some $u \in \uc$
      \end{enumerate}
      \textbf{then} add the edge $[\ell]$---$[\Pi]$ to $\ec_1$. 
    \end{enumerate}
  \item \textbf{Return} $(\nc_1, \ec_1)$.
  \end{enumerate}
\end{algo}
\begin{proof}
  In \ref{nodes}, we calculate $\nc_1$.
  The set $\gc.\pc_2$ contains representatives for all lines in $L \otimes \QQ$ up to $G_1$-equivalence, which we refine by identifying equivalent classes.
  If $[\ell_i], [\ell_j] \in \pc_2$ and $\ell_i \not \sim_{G_2} \ell_j$ then $\ell_i \not \sim_{G_1} \ell_j$ and so it suffices to establish conditions for $g g_i \ell_i = g_j \ell_i$ where $g_i, g_j \in \gc$ and $g \in G_1$.
  By an elementary calculation, the condition $g_j^{-1} g g_i \ell_i = \ell_i$ for some $g \in G_1$ is equivalent to $G_1 \cap \opn{Stab}_{G_2}(g_j \ell_i) g_j g_i^{-1} \neq \emptyset$, which is equivalent to $l_k g_j g_i^{-1} \in G_1$ for some $l_k \in \opn{Stab}_{G_1}(g_j \ell_i) \backslash \opn{Stab}_{G_2}(g_j \ell_i)$.
  The case of $\cc_1$ follows by an identical argument.
We calculate $\ec_1$ in \ref{edges}.
  As $Y$ contains representatives for all $\opn{Stab}_{G_1}(\Pi)$-orbits of lines in $\Pi$, we add the edge $[\ell]$---$[\Pi]$ to $\ec_1$ if and only if $\ell = g y_i$ for some $g \in G_1$ and $y_i \in Y$ or, equivalently, if $\ell = g \tau(y_i, \ell)^{-1} \ell$ for some $\tau(y_i, \ell) \in G_2$ where $\tau(y_i, \ell): y_i \mapsto \ell$.
  As this condition is equivalent to $G_1 \cap \opn{Stab}_{G_2}(\ell) \tau(y_i, \ell) \neq \emptyset$, the result follows.  
\end{proof}
\noindent \textbf{Remark} 
\begin{enumerate}
\item One can calculate  $G_1$-orbits of isotropic vectors in $L$ by performing steps \ref{init}-\ref{nodes} in Algorithm \ref{TitsAlgo} with $\pc_2$ replacing $\nc_2$ in \ref{nodes}.
\item Given a series of groups 
$G_1 \subset \ldots \subset G_m \subset \opn{O}^+(L \otimes \QQ)$,
  it is typically faster to calculate $\bc(G_1)$ by calculating each of the intermediate buildings $\bc(G_i)$ than working directly with $G_1 \subset G_m$.
\end{enumerate}
\subsection{Maximal lattices}
We now show that if $G_1 \subset \opn{O}^+(L)$ then one can always take $G_2 = \opn{O}^+(L')$ in Algorithm \ref{TitsAlgo}, where $L \subset L'$ is a maximal overlattice of $L$.
In such a case, $\bc(G_2)$ can be described using the results of  Attwell-Duval \cite{Attwell-Duval-Number, Attwell-Duval, Attwell-Duval-Topological}.
\begin{defn} A lattice $L'$ is said to be \textbf{maximal} if $D(L')$ contains no non-trivial totally isotropic subgroup
  (where a subgroup $H$ is said to be totally isotropic if $q_L \vert H \equiv 0 \bmod{ 2 \ZZ}$).
  An overlattice $L \subset L'$ with $L'$ maximal is said to be a \textbf{maximal overlattice}.
\end{defn}
\begin{lemma}\label{maximal-overlattice-lemma}
  Every lattice $L$ admits a maximal overlattice $L \subset L'$.
\end{lemma}
\begin{proof}
  As explained in \S4 of \cite{Nikulin}, overlattices $L \subset L'$ are in bijection with isotropic subgroups $H \subset D(L)$.
  We enlarge the initial set $H := \{0 \}$ by adding successive  isotropic elements in $H^{\perp}$ until $H^{\perp}/H$ contains no non-trivial isotropic subgroup.
  As $q_{L'} \cong H^{\perp}/H$ \cite{Nikulin} then $L'$ is maximal.
\end{proof}
\begin{defn}
  A maximal lattice $L'$ of signature $(2,n)$ is said to be \textbf{split} if
  \begin{equation}\label{maximalsplitting}
  L' \cong 2U \op L_0'
  \end{equation}
  for a lattice $L_0'$.
  We let $\Sc(L_0') = \{x_i\}_{i=1}^{n+2}$ denote a $\ZZ$-basis of $L'$ such that $\{x_1, x_2\}$, $\{x_3, x_3 \}$ are canonical bases for the two copies of $U$ in \eqref{maximalsplitting} and $\{x_i\}_{i=5}^{n+2}$ is a basis for $L_0'$.
  We write $\Sc(L)$ to denote $\Sc(L_0')$ for an arbitrary choice of $L_0'$ satisfying \eqref{maximalsplitting}.
\end{defn}
It is known that if $L'$ is maximal of signature $(2,n)$  then $L'$ splits whenever $n \geq 5$ \cite{Attwell-Duval}.
Split maximal lattices are also  very common when $n < 5$, as can be established from \cite[Corollary 1.13.5]{Nikulin}.
\\
\\
\noindent \textbf{Assumption}  
From now on, we will assume that all maximal lattices are of signature $(2,n)$ and split.
\\
\\
As explained in \cite[\S5]{Attwell-Duval}, a totally isotropic sublattice $E$ of a maximal lattice $L'$ defines a $\ZZ$-basis $\Sc(E)  = \{x_i  \}$ of $L'$ where, if  
$E = \la e_1, e_2 \ra$ then
$\{x_1, x_2 \}$,
$\{x_3, x_4 \}$
are canonical bases for orthogonal copies of $U$ and $x_1 = e_1$ and $x_2=e_3$;
if $E = \la e_1 \ra$ then $x_1=e_1$ \cite{Attwell-Duval}.
Furthermore, the totally isotropic sublattices $E_1$, $E_2 \subset L'$ belong to the same $\opn{O}^+(L')$-orbit if and only if $E_1^{\perp}/E_1 \cong E_2^{\perp}/E_2$ \cite{Attwell-Duval}.
Rephrasing the results of \cite{Attwell-Duval}, we obtain Algorithm \ref{MaximalAlgo} and Figure \ref{MaximalBuilding}.
\begin{algo}\label{MaximalAlgo}
  If $L'$ is a maximal lattice of signature $(2, n)$ with $n \geq 2$ then $\bc(\old)$ (illustrated in Figure \ref{MaximalBuilding}) can be calculated as follows.
  \begin{enumerate}
  \item \textbf{Let} $\cc:=\emptyset$ and \textbf{let} $\pc:=\{e\}$ where $e \in L'$ is a primitive isotropic vector.
  \item Calculate $\opn{gen}(0, n-2, D(L))$ (e.g. by \cite{largeclass}).
  \item \textbf{For} each class $L_0$ in $\opn{gen}(0, n-2, D(L))$
    \begin{enumerate}
    \item \textbf{Let} $\{x_i\}_{i=1}^{2+n}:=\Sc(L_0)$ and \textbf{let} $\Pi := \la x_1, x_3 \ra$.
    \item \textbf{Let} $\cc:=\cc \cup [\Pi]$.
    \end{enumerate}
  \item \textbf{For} $[\Pi] \in \cc$
    \textbf{let} $\ec:=\ec \cup \{ [e] \text{---}[\Pi]\}$.
\item \textbf{Return} $\bc(\opn{O}^+(L')) := (\pc \sqcup \cc, \ec)$.
  \end{enumerate}
\end{algo}
\begin{figure}
  \centering
  \begin{tikzpicture}

    \draw
    
    (0,0) --  (-4,-4) 
    (0,0) -- (-2,-4)
    (0,0) -- (2,-4) 
    (0,0) -- (4,-4) 
    
    (0,0.1) node[above]{$0$}
    (-4,-4.1) node[below]{$0$}
    (-2,-4.1) node[below]{$1$}
    (0, -4.1) node{$\ldots$}
    (2,-4.1) node[below]{${m-1}$}
    (4,-4.1) node[below]{$m$};
    \draw[black, fill=black] (0,0) circle (0.1cm);

        \filldraw[fill=white]
(-4,-4) circle (0.1)
    (-2,-4) circle (0.1)
    (2,-4) circle (0.1)
    (4,-4) circle (0.1);

  \end{tikzpicture}
  \caption{$\bc(\old)$ where $L'$ is maximal}\label{MaximalBuilding}
\end{figure}
\subsection{Generators and cosets}
We now exhibit generators for $G_2$, $\opn{Stab}_{G_2}(\ell)$ and $\opn{Stab}_{G_2}(\Pi)$ where $G_2 = \opn{O}^+(L')$ and $L'$ is a maximal lattice.
These generators can be used to calculate representatives for each of the cosets in Algorithm \ref{TitsAlgo}. 
\begin{defn}[{\hspace{1sp}\cite{Eichler, abelianisation}}]
  Let $L=2U \op L_0$ be a lattice of signature $(2,n)$ and, for primitive isotropic $e \in L$, let $a \in e^{\perp}$.
  Then there exists an element $t(e,a) \in \widetilde{\opn{SO}}^+(L)$, known as an \textbf{Eichler transvection}, defined by
  \begin{equation}\label{eichlerdefeq}
  t(e,a): v \mapsto v - (a,v)e + (e,v)a - \frac{1}{2}(a,a)(e,v)e,
  \end{equation}
where $v \in L$.
\end{defn}
\begin{lemma}[{\hspace{1sp}\cite[Proposition 3.3]{abelianisation}}]\label{2ULgensLemma}
  Suppose $L=U \op L_1$ is a lattice where $L_1=U \op L_0$.
  If $\{x_i\}_{i=1}^{n+2}$ is a $\ZZ$-basis for $L$ such that   $\{x_1, x_2\}$ is a canonical basis for $U$ 
  then, 
    \beq
    \opn{O}^+(L) = \la t(x_1, v), t(x_2, v), \opn{O}^+(L_1) \mid v \in L_1 \ra.  
    \eeq
\end{lemma}
\begin{lemma}\label{ParaDecompLemma}
  Let $L'$ be a maximal lattice.
  If $\ell \subset L'$ is an isotropic line, let $\{x_i\}_{i=1}^{n+2}:=\Sc(\ell)$ and $L_1':=\la x_i \mid i=3, \ldots, n \ra$.
  Then 
  \beq
  \opn{Stab}_{\old}(\ell) =
  \la t(x_1, v), \opn{O}^+(L_1') \mid v \in L_1' \ra.
  \eeq
\end{lemma}
\begin{proof}
  By \cite[p.30]{Attwell-Duval-Topological},  
  \beq
  \opn{Stab}_{\opn{O}^+(L)}(\ell) \cong \opn{O}^+(L_1) \ltimes U(\ell),
  \eeq
  where the unipotent radical $U(\ell) \subset \opn{Stab}_{\opn{O}^+(L)}(\ell)$ is generated by matrices of the form
  \begin{equation}\label{ZU1}
    g(\underline{z})=
    \left(
  \begin{array}{c c | c c c }
    1 & y_{1,2} & y_{1,3} & \ldots & y_{1,n+2} \\
    0 & 1 & 0 & \ldots & 0 \\
    \hline
    0 & z_1 & \\
    \vdots & \vdots & & \textbf{\huge{$I$}} & \\
    0 & z_n & 
  \end{array}
  \right) .
  \end{equation}
  The terms $y_{1, i}$ in \eqref{ZU1} satisfy 
  \begin{center}
    \begin{tabular}{c c c}
      $y_{1,2} = -(\underline{z}, \underline{z})$
      & and &
      $y_{1,i+2} = - (\underline{z}, x_i)$
      \end{tabular}
  \end{center}
  where $\underline{z} = \sum_{i=1}^n z_i x_{i+2}$ and $(z_1, \ldots, z_n) \in \ZZ^n$ \cite{Attwell-Duval-Topological}.
  By \eqref{eichlerdefeq},  $t(x_1, a+b) = t(x_1, a) t(x_1,b)$ for all $a,b \in L_1$, implying $g(\underline{z}) = t(x_1, \underline{z})$, from which the result follows.
\end{proof}
\begin{lemma}\label{stab-E-o(L')-gens}
  If $\Pi \subset L'$ is a totally isotropic plane and $\Sc(\Pi) = \{x_i\}_{i=1}^{n+2}$, let
  \beq
  L_0':=\la x_i \mid i=5, \ldots, n \ra.
  \eeq
  Then 
  \beq
  \opn{Stab}_{\opn{O}^+(L')}(\Pi) =
  \la
  t(x_1, x_4),
  t(x_2, x_3),
  t(x_1, x_3),
  t(x_1, v),
  t(x_3, v),
  \opn{O}^+(L_0') \mid v \in L_0'
  \ra.
   \eeq
\end{lemma}
\begin{proof}
  As the Jacobi group $\Gamma^J(L')$ is the subgroup of $\opn{Stab}_{\opn{O}^+(L')}(\Pi)$ acting trivially on $L_0$ \cite{abelianisation}, the result follows by noting that   \cite[p.470]{abelianisation}
  \beq
  \Gamma^J(L')
  =
  \la 
  t(x_1, x_4),
  t(x_2, x_3),
  t(x_1, x_3), 
  t(x_1, v), 
  t(x_3, v) \mid 
   v \in L_0 \ra.
  \eeq
\end{proof}
\noindent \textbf{Remark} 
 One can often calculate generators for the groups
$\opn{O}^+(L_1)$ and $\opn{O}^+(L_0)$ in Lemmas \ref{2ULgensLemma}, \ref{ParaDecompLemma} and \ref{stab-E-o(L')-gens} by using \cite{Mertens, Vinberg, FinckePohst}.
Furthermore, if $\Pi \subset L$ is a totally isotropic sublattice, one can also use Lemma \ref{stab-E-o(L')-gens} to understand the boundary curves $C_{\Pi}$ in Theorem \ref{BailyBorelThm}.
By \cite{Brieskorn}, there exists a homomorphism
\beq
\pi:\opn{Stab}_{\Gamma}(E) \rightarrow \opn{SL}(2, \ZZ)  
\eeq
and, by \cite[\S2]{Scattone}\cite[\S2]{Kondo}, $C_{\Pi} \cong \HH^+ / \opn{Im} \pi$.
\begin{lemma}\label{fin-index-lemma}
  The inclusions
  \beq
  \begin{cases}
    \otl \subset \old \\
    \opn{Stab}_{\otl}(\ell) \subset \opn{Stab}_{\old}(\ell) \\
    \opn{Stab}_{\otl}(\Pi) \subset \opn{Stab}_{\old}(\Pi)
  \end{cases}
  \eeq
    are of finite index.
\end{lemma}
\begin{proof}
  The following argument is used in \cite{Kondo} for moduli spaces of K3 surfaces, where it is attributed to O'Grady (see also \cite{DawesFamily}).
  If $M$ is the exponent of the group $L'/L$ then
  \beq
  L'(M) \subset L \subset L' 
  \eeq
  and we let $Q:=L'/L'(M)$.
  If $I := L /L'(M) \subset Q$ then $\opn{Stab}_{\old}(I) \subset \ol$ and, by the Orbit-Stabiliser theorem, $\opn{Stab}_{\old}(I) \subset \old$, implying $\vert \otl : \old \vert < \infty$. 
 The other cases follow by an identical argument.
\end{proof}
\noindent \textbf{Remark} By Schreier's Lemma \cite[p.18]{Cameron}, Lemma \ref{2ULgensLemma} can be used to obtain a set of generators for the subgroups occurring in Lemma \ref{fin-index-lemma}.
\subsection{The element $\tau(x,y)$}
If $L$ is a split maximal lattice and $\opn{O}^+(L) \subset G_2 \subset \opn{O}^+(L \otimes \QQ)$, then the element $\tau(x,y)$ in Algorithm \ref{TitsAlgo} can be calculated using Algorithm \ref{EichlerAlg}.
\begin{algo}\label{EichlerAlg}
  Let $L$ be a maximal lattice such that $L=U \op L_1$ where $L_1 = U \op L_0$.
  If $x, y \in L$ are primitive and isotropic then there exists $\tau(x,y) \in \opn{O}^+(L)$ such that $\tau(x,y) x = y$.
  The element $\tau(x,y)$ can be calculated as follows.
  \begin{enumerate}
  \item \textbf{Let} \label{eichler1}
\begin{equation}\label{SL2emb}
    \theta:\opn{SL}(2, \ZZ) \times \opn{SL}(2, \ZZ) \rightarrow \opn{O}^+(L)
    \end{equation}
    be defined by
    \begin{center}
      \begin{tabular}{c c c  }
        $\theta(Z, I) =
        \bpm
        d & 0 & c & 0 \\
        0 & a & 0 & -b \\
        b & 0 & a & 0 \\
        0 & -c & 0 & d
        \epm \op I$
        & and &
        $\theta(I, Z)=
        \bpm
        d & 0 & 0 & -c \\
        0 & a & b & 0 \\
        0 & c & d & 0 \\
        -b & 0 & 0 & a 
        \epm
        \op I$
      \end{tabular}
    \end{center}
    for 
    \beq
    Z = \bpm a & b \\ c & d \epm \in \opn{SL}(2, \ZZ).
    \eeq
  \item \textbf{Let} $\iota: L \rightarrow M_2(\ZZ)$ be defined by by \label{eichler2}
    \beq
    \iota: (x_1, x_2, \ldots ) \mapsto
    \bpm
    x_3 & -x_2 \\
    x_1 & x_4
    \epm.
    \eeq
\item   \textbf{If}  $w \in L_1$ \textbf{then}  
    \label{eichler3}
    \begin{enumerate}
    \item[]  \textbf{let} $g(w):=I \in \opn{O}^+(L)$
    \item[] \textbf{else}, by calculating the Smith normal form, \textbf{let} $(A, B) \in \opn{SL}(2, \ZZ) \times \opn{SL}(2, \ZZ)$ be such that $\iota(w)$ is diagonal and \textbf{let} $g(w):=\theta(A, B)$.
    \end{enumerate}
  \item  \textbf{Let} $x':=g(x) x$ and $y':=g(y) y$. \label{eichler4}
  \item \textbf{Let} $u', v' \in L_1$ be such that $(x',u')=(y',v')=1$ (which can be calculated using the Euclidean algorithm). \label{eichler5}
  \item \textbf{Return} $\tau(x,y)$ where \label{eichler6}
\beq
\tau(x,y):=g(x) t(e, u') t(f, x'-f') t(e, -v') g(y)^{-1}.
\eeq
  \end{enumerate}
\end{algo}
\begin{proof}
  The algorithm essentially follows the proof of the Eichler criterion given in  \cite[Proposition 3.3]{abelianisation}: we have simply made a few details explicit.
  The homomorphism \eqref{SL2emb} in \ref{eichler2} is well known (e.g. \cite{Sterk}) and is obtained by taking coordinates $(w,x,y,z)$ for $2U$ and identifying
  \beq
  2U \ni (x_1,x_2,x_3,x_4) \mapsto
  \bpm 
  x_3 & - x_2 \\
  x_1 & x_4
  \epm \in M_2(\ZZ),
  \eeq
  where the quadratic form on $M_2(\ZZ)$ is given by $2 \opn{det}$.
  The action of $(A,B) \in \opn{SL}(2, \ZZ) \times \opn{SL}(2, \ZZ)$ is then given by
  \beq
  \opn{SL}(2, \ZZ) \times \opn{SL}(2, \ZZ) \ni (A, b) \mapsto A X B^{-1}
  \eeq
  for $X \in M_2(\ZZ)$.
  The elements $u'$ and $v'$ in \ref{eichler5} both exist as $x'$ and $y'$ are primitive and $L'$ is maximal, implying $u'^*$ and $v'^*$ modulo $L$ are of order 1, hence $\opn{div}(x') = \opn{div}(y')=1$.
\end{proof}
\subsection{The sets $X$, $Y$ and the group $\Gamma$}
We now show how to calculate the set $X$, $Y$ and the group $\Gamma$ in Algorithm \ref{TitsAlgo}.
\begin{lemma}\label{X-Gamma-Lemma}
  Let $L \subset L'$ be a maximal overlattice and let 
  $\Pi = \la e_1, e_2 \ra$ be a totally isotropic plane.
  Then,
  \begin{enumerate}
  \item $\opn{SL}(2, \ZZ) \subset \opn{Stab}_{\old}(\Pi)$.
  \item If $L'(M) \subset L \subset L'$ for some $M>0$ and  $N$ is the  exponent of the group $L^{\vee}/ L'(M)$  then $\Gamma(N) \subset \opn{Stab}_{\otl}(\Pi)$.
  \item If $\hc$ is a  transversal for $\Gamma(N) \backslash \opn{SL}(2, \ZZ)$ then the set $\hc .e_1$ contains  representatives for all $\opn{Stab}_{\otl}(\Pi)$-orbits of isotropic lines in $\Pi$.
  \end{enumerate}
\end{lemma}
\begin{proof}
  We fix a basis $\Sc(\Pi)$ for $L'$.
  By \eqref{SL2emb}, the group $\opn{Stab}_{\opn{O}^+(L')}(\Pi)$ contains a copy of $\opn{SL}(2, \ZZ)$, acting as
  \begin{align*}
      \left(
  \begin{array}{c c  c c | c c c }
    d & 0 & c & 0 & & & \\
    0 & a & 0 & - b & &\textbf{\huge{$0$}} &\\
    b & 0 & a & 0 & & & \\
    0 & -c & 0 & d & & &\\
    \hline
    & &    &     &   & & \\
    & &  \textbf{\huge{$0$}}  &     & &  \textbf{\huge{$I$}}  &
  \end{array}
  \right) && \text{for} &&
  \bpm a & b \\ c & d \epm \in \opn{SL}(2, \ZZ)
  \end{align*}
and, by Lemma \ref{fin-index-lemma}, there exists $M$ such that
  \beq
  L'(M) \subset L \subset L' \subset L^{\vee}.
  \eeq
  Therefore, $\Gamma(N)$ acts trivially on $L^{\vee}/L'(M)$ and  $ \Gamma(N) \subset \otl$.
  The result follows by noting that $\opn{SL}(2, \ZZ)$ acts transitively on isotropic vectors in $\Pi$.  
\end{proof}
\subsection{Tits' buildings of enclosing groups}
If $G_1 \subset G_2$ then (as noted in \cite{abeliancompact} in connection with the moduli of abelian surfaces) $G_2$ acts naturally on $\bc(G_1)$.
If $\bc(G_1)$ is known, one can calculate $\bc(G_2)$ using Algorithm \ref{SuperGroupAlg}.
This can be helpful when working with groups of the form $\opn{O}_{\Ac}^+(L)$.
In such a case, one may be unable to apply Algorithm \ref{TitsAlgo} directly with $G_2=\opn{O}^+(L')$ for a split maximal lattice $L'$, however as $\otl \subset \opn{O}_{\Ac}^+(L)$ and $\otl \subset \opn{O}^+(L')$, one can calculate $\bc(\opn{O}_{\Ac}^+(L))$ using Algorithm \ref{MaximalAlgo} and \ref{SuperGroupAlg}.
\begin{algo}\label{SuperGroupAlg}
  Let
  $G_1 \subset G_2 \subset \opn{O}^+(L)$ and suppose $G_1 \subset \opn{O}^+(L')$ where $L \subset L'$ is an overlattice and both $\vert \opn{O}^+(L):G_1 \vert < \infty$ and
  $\vert \opn{O}^+(L'):G_1\vert < \infty$.
  Let $\bc(G_1) = (\nc, \ec)$ where $\nc = \pc \sqcup \cc$ and let $[E]$ denote the $G_1$-equivalence class of an isotropic subspace $E \subset L \otimes \QQ = L' \otimes \QQ$. 
Then $\bc(G_2)$ can be calculated from $\bc(G_1)$ by identifying nodes as follows.
  \begin{enumerate}
  \item \textbf{Let} $\hc$ be a transversal of $G_1 \backslash G_2$.
  \item \textbf{For} $\ic \in \{\pc, \cc\}$ 
      \begin{enumerate}
      \item[] \textbf{For} $h \in \hc$ and $([E_1], [E_2]) \in \ic \times \ic$
        \begin{enumerate}
        \item[] \textbf{If} there exists $\hat{g} \in \opn{O}^+(L')$ such that $\hat{g}(hE_1) = E_2$ \textbf{then} 
          \begin{enumerate}
          \item[] \textbf{Let} $\jc$ be a transversal for $\opn{Stab}_{\opn{O}^+(L')}(h E_1) \backslash \opn{Stab}_{G_1}(h E_1)$.
          \item[] \textbf{If} $\hat{g} \jc \cap G_1 \neq \emptyset$ \textbf{then} identify $[E_1]$ and $[E_2]$.
          \end{enumerate}
        \end{enumerate}
      \end{enumerate}
  \end{enumerate}
\end{algo}
\begin{proof}
  The classes $[E_1], [E_2] \in \pc$ or $\cc$ are equivalent under $G_2$ if and only if $ghE_1 = E_2$ for some $g \in G_1$ and $h \in \hc$.
  In such a case, there exists $x \in \opn{O}^+(L')$ such that $x:hE_1 \rightarrow E_2$.
  The set
  \beq
  X=\left \{ x \in \opn{O}^+(L') \mid x(hE_1) = E_2 \right \}
  =
  \hat{g} \opn{Stab}_{\opn{O}^+(L')}(hE_1)
  \eeq
  and $X \cap G_1 \neq \emptyset$ if and only if $\hat{g} \jc \cap G_1 \neq \emptyset$, from which the result follows.
\end{proof}
\noindent \textbf{Remark} We note that if $L'$ is a split maximal lattice, the map $\hat{g}$ in Algorithm \ref{SuperGroupAlg} can be calculated using \eqref{maximalsplitting} and Algorithm \ref{EichlerAlg}.

 \section{Examples}\label{ExampleSection}
We now calculate $\bc(\otl)$ for
$L=2U \op A_2$,
$L=2U \op \la -6 \ra \op \la -2 \ra$ and
$L=2 U(2) \op A_2$.
We will use a computer for lengthier parts of the calculations: our code (written for the \ttt{Sage} computer algebra system) can be found at \cite{boundary}.
Throughout, we will use $\{e_i, f_i \}$ for $i=1,2$ to denote canonical bases for the two copies of $U \subset L$.
\subsection{$L=2U \op A_2$}
Let $L=2U \op A_2$. 
We begin by calculating generators for $\opn{O}^+(L_1)$, where $L_1=U \op A_2$.
 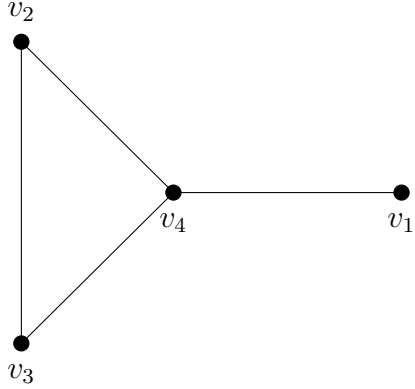
\begin{figure}[h!]
    \centering
    \begin{tikzpicture}
      \filldraw
      (0,2)  circle (0.1) 
      (0,-2) circle (0.1) 
      (2,-0) circle (0.1) 
      (5,-0) circle (0.1); 
      \draw
      (0,2.4)  node[align=center]{$v_2$}
      (0,-2.4) node[align=center]{$v_3$}
      (2,-0.4) node[align=center]{$v_4$}
      (5,-0.4) node[align=center]{$v_1$};
      \draw
      (0,2)--(0,-2)--(2,0)--cycle;
      \draw
      (2,0)--(5,0);
    \end{tikzpicture}
    \caption{A fundamental domain for $W(U \op A_2)$.}\label{U+A2diagram}
  \end{figure}
\begin{example}\label{O+(U+A2)}
  If $\{e_i\}_{i=1}^4$ is a $\ZZ$-basis for $L_1$ with Gram matrix
  \beq
  ((e_i, e_j))=
  \bpm
  0 & 1 & 0 & 0 \\
  1 & 0 & 0 & 0 \\
  0 & 0 & -2 & -1 \\
  0 & 0 & -1 & -2
  \epm
  \eeq
  then $\opn{O}^+(L_1)$ is generated by the reflections
  $\{\sigma_{v_i}\}_{i=1}^4$
  and the map
  \begin{equation}\label{diagaut}
  U \op A_2 \ni (x_1, x_2, x_3, x_4) \mapsto (x_1, x_2, -x_4, -x_3) \in U \op A_2,
  \end{equation}
  where
  $v_1 = (1,-1,0,0)$, 
  $v_2 = (0,0,-1,0)$, 
  $v_3 = (0,0,0,1)$ and 
  $v_4 = (0,1,1,-1)$.
\end{example}
\begin{proof}
  We begin by applying Vinberg's algorithm \cite{Vinberg} to obtain generators $\{\sigma_{v_i} \}$ for the reflection subgroup $W(L_1) \subset \opn{O}^+(L_1)$.
  Let $\HH(L_1)$ be a component of $\{[x] \in \PP(L_1 \otimes \CC) \mid (x,x)<0 \}$ and, for fixed $x_0 \in \HH(L_1)$, let $H(v_i)^-$ denote the half-space of $\HH(L_1)$ orthogonal to $v_i \in L_1$ containing $x_0$.
  Vinberg's algorithm proceeds by selecting a point $x_0 \in \HH(L_1)$ whose stabiliser is generated by reflections $\{\sigma_{v_i} \}_{i=1}^k$ and proceeds by adding additional reflections $\sigma_{v_j}$, one-by-one for $j>k$, such that  
\begin{equation}\label{vinquant}
  \frac{\vert (x_0, v_j) \vert^2}{\vert (v_j, v_j) \vert}
\end{equation}
is minimal and $H(v_i)$ and $H(v_j)$ are opposite for all $i<j$ \cite[p.327]{Vinberg}.
The algorithm terminates at $j=m$ if the polyhedron $P^{(m)}$, defined by
\begin{equation*}
  P^{(m)} = \bigcap_{i=1}^m H(v_i)^-,
\end{equation*}
has finite volume (for which sufficient criteria are given in \cite{Vinberg}).

If $x_0 := e_1 + e_2$ then 
$x_0^{\perp} = \la e_1 - e_2, e_3, e_4 \ra \subset U \op A_2(-1)$
and  $x_0^{\perp} \cong \la -2 \ra \op A_2(-1)$.
Therefore, $\opn{Stab}_{W(L_1)}(x_0) = \la \sigma_{v_1}, \sigma_{v_2}, \sigma_{v_3} \ra$ where
$v_1 = e_1 - e_2$,
$v_2 = -e_3$
and $v_3 = e_4$.
If $v \in U \op A_2(-1)$ and $\sigma_v \in \opn{O}(L_1)$ then $v^2 \mid 2 \opn{div}(v)$.
As $\opn{div}(v)$ is the order of $v^*$ in $D(L_1)$ then $\opn{div}(v) = 1$ or $3$ and 
\eqref{vinquant} is minimised for $v^2 = -2$ or $-6$, respectively.
If $v_4 := e_2 + e_3 - e_4$ then $H(v_4)^-$ is opposite to $H(v_i)^-$ for all $i<4$ and $P^{(4)}$ is non-degenerate with finite volume.
The Coxeter diagram of $P^{(4)}$ (which was also obtained in  \cite{VinbergReflection}) is given in Figure \ref{U+A2diagram}.
The result follows by noting that the group of diagram automorphisms $D$ of $P$ is generated by \eqref{diagaut} and $\opn{O}(L) = W(\opn{O}(L)) \rtimes D$ \cite{Vinberg}. 
\end{proof}

\begin{example}
  The building $\bc(\otl)$ is given by Figure \ref{stable_2U_A2_tits}.
\end{example}
\begin{proof}
  We apply Algorithm \ref{TitsAlgo} with $G_1=\otl$ and $G_2=\opn{O}^+(L)$.
  The lattice $L$ is maximal and $A_2$ is unique in its genus \cite[Table 15.1 p.360]{SPLAG}.
  Therefore, $\bc(G_1)$ is as in Figure \ref{MaximalBuilding}, with $m=0$ and nodes $\pc:=\{ [\ell] \}$ and $\cc:=\{ [\Pi] \}$,
  where $\ell = \la e_1 \ra$ and $\Pi = \la e_1, e_2 \ra$.
  Using the generators $\ic$ of Lemma \ref{2ULgensLemma} and Example \ref{O+(U+A2)}, we calculate $\gc=G_2 / G_1$.
  We begin with an initial set $\gc := \{e \}$, which we enlarge by adding $x \in \gc. \ic$ to $\gc$ if there exists no $y \in \gc$ such that $y^{-1}x \in G_1$, and stop when $\gc$ can be enlarged no further.
  We find $\gc = \{g_1, g_2 \}$, where 
\beq
  \begin{array}{ccc}
    g_1:=\bpm
    1 & 0 & 0 & 0 & 0 & 0 \\
    0 & 1 & 0 & 0 & 0 & 0 \\
    0 & 0 & 1 & 0 & 0 & 0 \\
    0 & 0 & 0 & 1 & 0 & 0 \\
    0 & 0 & 0 & 0 & 1 & 0 \\
    0 & 0 & 0 & 0 & 0 & 1
    \epm
    & \text{and} &
    g_2:= \bpm
    1 & 0 & 0 & 0 & 0 & 0 \\
    0 & 1 & 0 & 0 & 0 & 0 \\
    0 & 0 & 1 & 0 & 0 & 0 \\
    0 & 0 & 0 & 1 & 0 & 0 \\
    0 & 0 & 0 & 0 & 0 & -1 \\
    0 & 0 & 0 & 0 & -1 & 0
    \epm.
  \end{array}
\eeq
Similarly (using the generators of Lemma \ref{ParaDecompLemma} and \ref{stab-E-o(L')-gens}) we find that $\{g_1, g_2 \}$ are also representatives for 
  $\opn{Stab}_{G_1}(g_i \ell) \backslash \opn{Stab}_{G_2}(g_i \ell)$
  and
  $\opn{Stab}_{G_1}(g_i \Pi) \backslash \opn{Stab}_{G_2}(g_i \Pi)$, 
implying $\nc_1 = (\{[\ell] \}, \{[\Pi]\})$.
  By Lemma \ref{X-Gamma-Lemma}, we let $X=\{\ell \}$ and $\jc = \opn{SL}(2, \ZZ)/\Gamma(3)$, and as $\ell \in Y:=G.X$ we add the edge $l$---$\Pi$ to $\ec_1$. 
\end{proof}
The calculation of $\bc(\otl)$ using \cite{boundary} is fast, taking less than a second on a desktop computer.
\begin{figure}
  \centering
  \begin{tikzpicture}
    \draw
    (-3,0)--(3,0)
    (-3, -0.1) node[below]{$0$}
(3, -0.1) node[below]{$0$};
    
    \filldraw[fill=white]
    (-3,0) circle (0.1);

    \filldraw[fill=black]
    (3,0) circle  (0.1);
  \end{tikzpicture}
  \caption{The building $\bc(\widetilde{\opn{O}}^+(2U \op A_2))$}\label{stable_2U_A2_tits}
\end{figure}
\subsection{$L=2U \op \la -6 \ra \op \la - 2 \ra$}
We now consider $L=2U \op \la -6 \ra \op \la - 2 \ra$, which arises in connection with period spaces of deformation generalised Kummer varieties \cite{DawesFamily}.
\begin{example}\label{MaximalLatticeExample}
  There exists a maximal overlattice $L \subset L'$.
\end{example}
\begin{proof}
  The discriminant form of $L$ is given by
  \beq
  q_L(a,b,c) =
  -\frac{a^2}{2}
  -\frac{3b^2}{2}
  -\frac{2c^2}{3} \bmod{2 \ZZ}
  \eeq
  where $(a,b,c) \in C_2^{\op 2} \op C_3 \cong D(L)$.
  The isotropic elements of $D(L)$ are given by $(1,1,0)$ and $(0,0,0)$.
  By Lemma \ref{maximal-overlattice-lemma},
  if $H := \la (1,1,0) \ra$ then $H^{\perp} = \la (1,1,0), (0,0,1) \ra$ and $H^{\perp}/H \cong q_{A_2}$.
  The overlattice can be realised by the map 
  \beq
  2U \op \la -6 \ra \op \la -2 \ra \ni
  (x_1, x_2, x_3, x_4, x_5, x_6) \mapsto
  (x_1, x_2, x_3, x_4, -2 x_5, x_6 + x_5)
  \in 2U \op A_2.
  \eeq
\end{proof}
\begin{example}
  The buildings $\bc(\widetilde{\opn{O}}^+(L))$ and $\bc(\opn{O}^+(L))$ are given by Figure \ref{gkbuilding}.
\end{example}
\begin{proof}
  We calculate $\bc(\widetilde{\opn{O}}^+(L))$ using \cite{boundary}, which applies Algorithm \ref{TitsAlgo} and Example \ref{MaximalLatticeExample}.
If $\ell = \la e_1 \ra$ and $\Pi=\la e_1, e_2 \ra$ then $\pc = \{g_1 \ell, g_2 \ell \}$ and $\cc = \{g_1 \Pi, g_2 \Pi \}$, where
\begin{align*}
  g_1 =
  \bpm
  1 & 0 & 0 & 0 & 0 & 0 \\
  0 & 1 & 0 & 0 & 0 & 0 \\
  0 & 0 & 1 & 0 & 0 & 0 \\
  0 & 0 & 0 & 1 & 0 & 0 \\
  0 & 0 & 0 & 0 & 1 & 0 \\
  0 & 0 & 0 & 0 & 0 & 1 
  \epm
  && \text{and} &&
  g_2 =
  \bpm
  1 & 0 & 0 & 0 & 0 & 0 \\
  1 & 1 & 0 & 0 & 2 & 1 \\
  0 & 0 & 1 & 0 & 0 & 0 \\
  0 & 0 & 0 & 1 & 0 & 0 \\
  1 & 0 & 0 & 0 & 1 & 0 \\
  0 & 0 & 0 & 0 & 0 & 1
  \epm.
\end{align*}
  As shown in \cite[Lemma 2.4]{DawesFamily}, $\opn{O}^+(L) = \la \widetilde{\opn{O}}^+(L), \sigma_v \ra$ where $v$ generates the $\la -6 \ra$ factor of $L$.
  We therefore apply Algorithm \ref{SuperGroupAlg} with $H = \{I, \sigma_v \}$.
  For  $\ic = \pc$, if $h = I$ we calculate a transversal $\{r_k\}$ for $\opn{Stab}_{\otl}(g_2 \ell) \backslash \opn{Stab}_{\ol}(g_2 \ell)$.
  The element $\hat{g}=g_2$ and we find  no $r_k \hat{g}^{-1} \in G_2$. 
  Similarly, for $h=\sigma_v$, we have $\hat{g}=g_2$ and the same conclusion follows.
  Therefore, $[g_1 \ell]$ and $[g_2 \ell]$ are not identified under $\opn{O}^+(L)$.
  One proceeds identically for $\ic = \Cc$.
\end{proof}
\begin{figure}
  \centering
  \begin{tikzpicture}
    \draw
    (-6,0)--(6,0)
    (-6, -0.1) node[below]{$0$}
    (-2, -0.1) node[below]{$0$}
    (2, -0.1) node[below]{$1$}
    (6, -0.1) node[below]{$1$};

    \filldraw[fill=white]
    (-6,0) circle (0.1)
    (2,0) circle (0.1);
    \filldraw[fill=black]
    (6,0) circle (0.1)    
    (-2,0) circle (0.1);

  \end{tikzpicture}
  \caption{The building $\bc(\opn{O}^+(2U \op \la - 2 \ra \op \la -6 \ra))$}\label{gkbuilding}
\end{figure}
\noindent \textbf{Remark}
Figure \ref{gkbuilding} was also obtained using different methods in \cite{DawesFamily}. 
\subsection{$2U(2) \op A_2$}
For a more intricate example, we consider $L=2U(2) \op A_2$.
\begin{figure}[h!]
  \centering
  \begin{tikzpicture}
    \draw circle (4cm);
    
    \path (3*360/12:4cm) edge [bend right=30] (9*360/12:4cm);	
    \path (11*360/12:4cm) edge [bend right=30] (5*360/12:4cm);	
    \path (7*360/12:4cm) edge [bend right=30] (1*360/12:4cm);
    
    \draw[black, fill=black] (2*360/12:1.17cm) circle (0.1cm);
    \node at (2*360/12:1.47cm) {3};
    \draw[black, fill=black] (6*360/12:1.17cm) circle (0.1cm);
    \node at (6*360/12:1.47cm) {1};
    \draw[black, fill=black] (10*360/12:1.17cm) circle (0.1cm);
    \node at (10*360/12:1.47cm) {6};
    
    \draw[black, fill=black] (0:4cm) circle (0.1cm);
    \draw[black, fill=black] (2*360/12:4cm) circle (0.1cm);
    \draw[black, fill=black] (4*360/12:4cm) circle (0.1cm);
    \draw[black, fill=black] (6*360/12:4cm) circle (0.1cm);
    \draw[black, fill=black] (8*360/12:4cm) circle (0.1cm);  
    \draw[black, fill=black] (10*360/12:4cm) circle (0.1cm);  
    
    \node at (0*360/12:4.3cm) {8};
    \node at (2*360/12:4.3cm) {7};
    \node at (4*360/12:4.3cm) {5};
    \node at (6*360/12:4.3cm) {2};
    \node at (8*360/12:4.3cm) {0};
    \node at (10*360/12:4.3cm) {4};
    
    \draw[black, fill=white] (1*360/12:4cm) circle (0.1cm);
    \draw[black, fill=white] (3*360/12:4cm) circle (0.1cm);
    \draw[black, fill=white] (5*360/12:4cm) circle (0.1cm);
    \draw[black, fill=white] (7*360/12:4cm) circle (0.1cm);
    \draw[black, fill=white] (9*360/12:4cm) circle (0.1cm);
    \draw[black, fill=white] (11*360/12:4cm) circle (0.1cm);
    
    \node at (1*360/12:4.3cm) {5};
    \node at (3*360/12:4.3cm) {4};
    \node at (5*360/12:4.3cm) {1};
    \node at (7*360/12:4.3cm) {3};
    \node at (9*360/12:4.3cm) {0};
    \node at (11*360/12:4.3cm) {2};
    
  \end{tikzpicture}
  \caption{$\bc(\widetilde{\opn{O}}^+(2U(2) \op A_2))$}\label{2u2a2building}
\end{figure}
\begin{example}
  The building $\bc(\otl)$ is given in Figure \ref{2u2a2building}.
  \end{example}
We calculated $\bc(\otl)$ using \cite{boundary}, applying Algorithm \ref{TitsAlgo} for $G_2 = \opn{O}^+(2U \op A_2)$ and $G_1 = \otl$.
As the index $\vert G_2:G_1 \vert = 1080$ is quite large, the calculation was rather lengthy, taking a few hours on a desktop computer.
A significant  improvement in performance should be possible by making use of the subgroups
\beq
\widetilde{\opn{O}}^+(2U(2) \op A_2)
\subset
\widetilde{\opn{O}}^+(U \op U(2) \op A_2)
\subset
\widetilde{\opn{O}}^+(2U  \op A_2)
\subset
\opn{O}^+(2U \op A_2),
\eeq
which are of index
20, 27 and 2, respectively.
 \small{
\bibliographystyle{amsalpha}
\bibliography{bib}{}

\newcommand{\etalchar}[1]{$^{#1}$}
\providecommand{\bysame}{\leavevmode\hbox to3em{\hrulefill}\thinspace}
\providecommand{\MR}{\relax\ifhmode\unskip\space\fi MR }
\providecommand{\MRhref}[2]{%
  \href{http://www.ams.org/mathscinet-getitem?mr=#1}{#2}
}
\providecommand{\href}[2]{#2}
\begin{thebibliography}{{Cam}99}

\bibitem[AD14]{Attwell-Duval-Number}
D.~Attwell-Duval, \emph{{On the number of cusps of orthogonal Shimura
  varieties}}, {Ann. Math. Qu\'e.} \textbf{38} (2014), no.~2, 119--131.

\bibitem[AD15]{Attwell-Duval}
\bysame, \emph{{Algebraic and Geometric Properties of the Boundary of
  Orthogonal Shimura Varieties}}, Ph.D. thesis, McGill University, 2015.

\bibitem[AD17]{Attwell-Duval-Topological}
\bysame, \emph{{Topological and algebraic results on the boundary of connected
  orthogonal Shimura varieties}}, {Ann. Math. Qu\'e.} \textbf{41} (2017),
  no.~1, 27--42.

\bibitem[BB66]{bailyborel}
W.L. Baily, Jr. and A.~Borel, \emph{Compactification of arithmetic quotients of
  bounded symmetric domains}, Ann. of Math. (2) \textbf{84} (1966), 442--528.

\bibitem[BCP97]{MAGMA}
W.~Bosma, J.~Cannon, and C.~Playoust, \emph{The {M}agma algebra system. {I}.
  {T}he user language}, J. Symbolic Comput. \textbf{24} (1997), no.~3-4,
  235--265, Computational algebra and number theory (London, 1993).

\bibitem[{Bri}83]{Brieskorn}
E.~{Brieskorn}, \emph{{Die Milnorgitter der exzeptionellen unimodularen
  Singularit\"aten}}, {Bonn. Math. Schr. 150}, 1983.

\bibitem[{Cam}99]{Cameron}
P.J. {Cameron}, \emph{{Permutation groups}}, vol.~45, Cambridge: Cambridge
  University Press, 1999.

\bibitem[Cas78]{cassels}
J.W.S. Cassels, \emph{Rational quadratic forms}, London Mathematical Society
  Monographs, vol.~13, Academic Press, Inc. London-New York, 1978.

\bibitem[CS99]{SPLAG}
J.H. Conway and N.J.A. Sloane, \emph{Sphere packings, lattices and groups},
  third ed., Grundlehren der Mathematischen Wissenschaften, vol. 290,
  Springer-Verlag, New York, 1999.

\bibitem[Daw]{DawesFamily}
M.R. Dawes, \emph{{The Baily-Borel compactification of a family of orthogonal
  modular varieties}}, Osaka Journal of Mathematics (to appear).

\bibitem[Daw22]{boundary}
\bysame, \emph{Buildings}, \url{http://github.com/m-dawes/buildings},
  21-05-2022.

\bibitem[Der04]{DernThesis}
T.~Dern, \emph{{H}ermitesche {M}odulformen zweiten {G}rades}, Aachen, 2004.

\bibitem[{Eic}52]{Eichler}
M.~{Eichler}, \emph{{Quadratische Formen und orthogonale Gruppen}}, vol.~63,
  {Springer-Verlag, Berlin-G\"ottingen-Heidelberg}, 1952.

\bibitem[FP85]{FinckePohst}
U.~{Fincke} and M.~{Pohst}, \emph{{Improved methods for calculating vectors of
  short length in a lattice, including a complexity analysis}}, {Math. Comput.}
  \textbf{44} (1985), 463--471.

\bibitem[GAP21]{GAP}
The GAP~Group, \emph{{GAP -- Groups, Algorithms, and Programming, Version
  4.11.1}}, 2021, {\tt https://www.gap-system.org}.

\bibitem[GHS07]{GHSK3}
V.~Gritsenko, K.~Hulek, and G.K. Sankaran, \emph{The {K}odaira dimension of the
  moduli of {$K3$} surfaces}, Invent. Math. \textbf{169} (2007), no.~3,
  519--567.

\bibitem[GHS09]{abelianisation}
V.A. Gritsenko, K.~Hulek, and G.K. Sankaran, \emph{Abelianisation of orthogonal
  groups and the fundamental group of modular varieties}, J. Algebra
  \textbf{322} (2009), no.~2, 463--478.

\bibitem[GHS13]{handbook}
\bysame, \emph{Moduli of {K}3 surfaces and irreducible symplectic manifolds},
  Handbook of moduli. {V}ol. {I}, Adv. Lect. Math. (ALM), vol.~24, Int. Press,
  Somerville, MA, 2013, pp.~459--526.

\bibitem[GKR13]{GKRBorcherdsproducts}
D.~{Gehre}, J.~{Kreuzer}, and M.~{Raum}, \emph{{Computing Borcherds products}},
  {LMS J. Comput. Math.} \textbf{16} (2013), 200--215.

\bibitem[HKW93]{abeliancompact}
K.~{Hulek}, C.~{Kahn}, and S.H. {Weintraub}, \emph{{Moduli spaces of abelian
  surfaces: compactification, degenerations, and theta functions}}, vol.~12,
  Berlin: Walter de Gruyter, 1993.

\bibitem[{Kne}56]{KneserIndef}
M.~{Kneser}, \emph{{Klassenzahlen indefiniter quadratischer Formen in drei oder
  mehr Ver\"anderlichen}}, {Arch. Math.} \textbf{7} (1956), 323--332.

\bibitem[Kne02]{Kneser}
M.~Kneser, \emph{{Quadratische Formen. Neu bearbeitet und herausgegeben in
  Zusammenarbeit mit Rudolf Scharlau.}}, Berlin: Springer, 2002.

\bibitem[Kon93]{Kondo}
S.~Kond{\=o}, \emph{On the {K}odaira dimension of the moduli space of {$K3$}
  surfaces}, {Compos. Math.} \textbf{89} (1993), no.~3, 251--299.

\bibitem[{Mer}14]{Mertens}
M.H. {Mertens}, \emph{{Automorphism groups of hyperbolic lattices}}, {J.
  Algebra} \textbf{408} (2014), 147--165.

\bibitem[{New}72]{morris}
M.~{Newman}, \emph{{Integral matrices}}, {Academic Press New York-London},
  1972.

\bibitem[{Nik}80]{Nikulin}
V.V. {Nikulin}, \emph{{Integral symmetric bilinear forms and some of their
  applications}}, {Math. USSR, Izv.} \textbf{14} (1980), 103--167.

\bibitem[PP85]{PleskenPohst}
W.~{Plesken} and M.~{Pohst}, \emph{{Constructing integral lattices with
  prescribed minimum. I}}, {Math. Comput.} \textbf{45} (1985), 209--221.

\bibitem[PS97]{PleskenSouvignier}
W.~{Plesken} and B.~{Souvignier}, \emph{{Computing isometries of lattices}},
  {J. Symb. Comput.} \textbf{24} (1997), no.~3-4, 327--334.

\bibitem[{Rau}11]{RaumImplement}
M.~{Raum}, \emph{{How to implement a modular form}}, {J. Symb. Comput.}
  \textbf{46} (2011), no.~12, 1336--1354.

\bibitem[S{\etalchar{+}}22]{SAGE}
W.A. Stein et~al., \emph{{S}age {M}athematics {S}oftware ({V}ersion 9.6)}, The
  Sage Development Team, 2022, {\tt http://www.sagemath.org}.

\bibitem[{Sca}87]{Scattone}
F.~{Scattone}, \emph{{On the compactification of moduli spaces for algebraic K3
  surfaces}}, vol. 374, Providence, RI: American Mathematical Society (AMS),
  1987.

\bibitem[SH98]{largeclass}
R.~{Scharlau} and B.~{Hemkemeier}, \emph{{Classification of integral lattices
  with large class number}}, {Math. Comput.} \textbf{67} (1998), no.~222,
  737--749.

\bibitem[{Ste}91]{Sterk}
H.~{Sterk}, \emph{{Compactifications of the period space of Enriques surfaces.
  I}}, {Math. Z.} \textbf{207} (1991), no.~1, 1--36.

\bibitem[Vin75]{Vinberg}
E.B. Vinberg, \emph{Some arithmetical discrete groups in {L}obacevskii spaces},
  Discrete subgroups of Lie groups and applications to moduli (1975), 328--348.

\bibitem[Vin07]{VinbergReflection}
\bysame, \emph{Classification of 2-reflective hyperbolic lattices of rank 4},
  Transactions of the Moscow Mathematical Society \textbf{68} (2007), 39--66.

\end{thebibliography}
}
\bigskip
\noindent Matthew Dawes\\
School of Mathematics\\
Cardiff University\\
Senghennydd Road\\
Cardiff\\
CF24 4AG\\
United Kingdom\\
\\
\noindent \ttt{dawesm1@cardiff.ac.uk}
\end{document}